\documentclass[11pt,reqno]{amsproc}
\allowdisplaybreaks
\numberwithin{equation}{section}

\usepackage[utf8]{inputenc}

\usepackage{cite}
\usepackage{color}
\usepackage{morefloats}
\usepackage{fullpage}
\usepackage{graphicx}
\usepackage{textcomp}
\usepackage{subfigure}
\usepackage{enumerate}
\usepackage[debug=false, colorlinks=true, pdfstartview=FitV,
linkcolor=blue, citecolor=blue, urlcolor=blue]{hyperref}
\usepackage[semicolon,square,authoryear]{natbib}

\usepackage[doipre={DOI:~}]{uri}

\newtheorem{theorem}{Theorem}[section]
\newtheorem{lemma}[theorem]{Lemma}
\newtheorem{proposition}[theorem]{Proposition}

\usepackage{morefloats}

\newlength{\drop}
\definecolor{amethyst}{rgb}{0.6, 0.4, 0.8}
\definecolor{burgundy}{rgb}{0.5, 0.0, 0.13}

\title{\textbf{On optimal designs using topology optimization for flow through porous media applications}}

\author{\textbf{T.~Phatak} and \textbf{K.~B.~Nakshatrala} \\
  {\small Department of Civil and Environmental
    Engineering, University of Houston, Texas. \\
    \textbf{Correspondence to:}~knakshatrala@uh.edu}}


\begin{document}

\begin{titlepage}
  \drop=0.1\textheight
  \centering
  \vspace*{\baselineskip}
  \rule{\textwidth}{1.6pt}\vspace*{-\baselineskip}\vspace*{2pt}
  \rule{\textwidth}{0.4pt}\\[\baselineskip]
       {\Large \textbf{\color{burgundy}
       On optimal designs using topology optimization for flow through porous media applications}}\\[0.3\baselineskip]
       \rule{\textwidth}{0.4pt}\vspace*{-\baselineskip}\vspace{3.2pt}
       \rule{\textwidth}{1.6pt}\\[\baselineskip]
       \scshape
       An e-print of the paper is available on arXiv. \par
       \vspace*{1\baselineskip}
       Authored by \\[\baselineskip]

  {\Large T.~Phatak\par}
  {\itshape Graduate student, University of Houston, Texas 77204.}\\[0.75\baselineskip]

  {\Large K.~B.~Nakshatrala\par}
  {\itshape Department of Civil \& Environmental Engineering \\
  University of Houston, Houston, Texas 77204. \\
  \textbf{phone:} +1-713-743-4418, \textbf{e-mail:} knakshatrala@uh.edu \\
  \textbf{website:} http://www.cive.uh.edu/faculty/nakshatrala}\\[2\baselineskip]
\begin{figure*}[h]
	\subfigure[Without enforcing radial symmetry]{\includegraphics[scale=0.36]{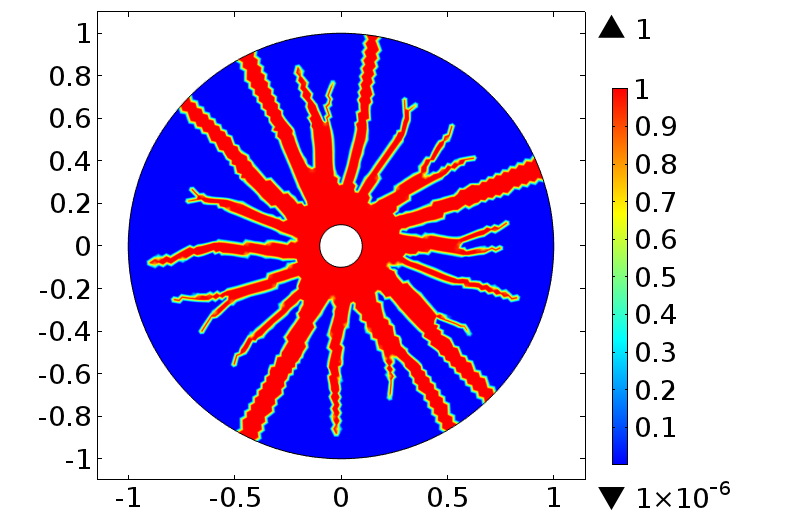}}
	\hspace{0.5cm}
	\subfigure[Invoking axisymmetry]{\includegraphics[scale=0.3]{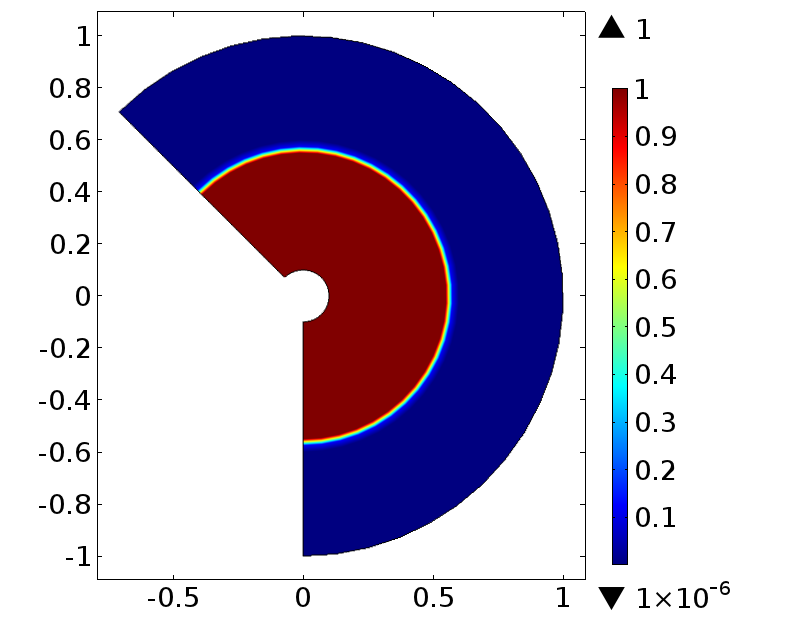}}

	\emph{We have considered a pressure-driven problem with axisymmetry and got optimal material layouts using topology optimization by maximizing the total rate of dissipation. The left figure shows unphysical finger-like design patterns when the primal analysis does not enforce explicitly the underlying radial symmetry. The right figure shows that one can avoid such numerical pathologies if the primal analysis invokes axisymmetry conditions.}
\end{figure*}

  \vfill
  {\scshape 2020} \\
  {\small Computational \& Applied Mechanics Laboratory} \par
\end{titlepage}

\begin{abstract}
Topology optimization (\textsf{TopOpt}) is a mathematical-driven design procedure to realize optimal material architectures. This procedure is often used to automate the design of devices involving flow through porous media, such as micro-fluidic devices. \textsf{TopOpt} offers material layouts that control the flow of fluids through porous materials, providing desired functionalities. Many prior studies in this application area have used Darcy equations for primal analysis and the minimum power theorem (\textsf{MPT}) to drive the optimization problem. But both these choices (Darcy equations and \textsf{MPT}) are restrictive and not valid for general working conditions of modern devices. Being simple and linear, Darcy equations are often used to model flow of fluids through porous media. However, two inherent assumptions of the Darcy model are: the viscosity of a fluid is a constant, and inertial effects are negligible. There is irrefutable experimental evidence that viscosity of a fluid, especially organic liquids, depends on the pressure. Given the typical small pore-sizes, inertial effects are dominant in micro-fluidic devices. Next, \textsf{MPT} is not a general principle and is not valid for (nonlinear) models that relax the assumptions of the Darcy model. This paper aims to overcome the mentioned deficiencies by presenting a general strategy for using  \textsf{TopOpt}. First, we will consider nonlinear models that take into account the pressure-dependent viscosity and inertial effects, and study the effect of these nonlinearities on the optimal material layouts under \textsf{TopOpt}. Second, we will explore the rate of mechanical dissipation, valid even for nonlinear models, as an alternative for the objective function. Third, we will present analytical solutions of optimal designs for canonical problems; these solutions not only possess research and pedagogical values but also facilitate verification of computer implementations.
\end{abstract}

\maketitle

\vspace{-0.25in}

\section*{ABBREVIATIONS}
1) \textsf{TopOpt}: topology optimization,
2) \textsf{MPT}:  minimum power theorem
\vspace{-0.1in}


\setcounter{figure}{0}   

\section{INTRODUCTION AND MOTIVATION}
Topology optimization---referred to as \textsf{TopOpt} hereon---is a contemporary leading design procedure for obtaining optimal material layouts. \textsf{TopOpt} being mathematical-driven, the whole design procedure is automatic with minimal or no human intervention \citep{bendsoe_sigmund_2013topology}. Another attractive feature is \textsf{TopOpt}'s ability to introduce new boundaries into the design space \citep{bendsoe_1995optimization, rozvany2014topology,nakshatrala2013nonlinear}. Said differently, \textsf{TopOpt} could provide designs with holes even if the user-provided initial configuration had none; for example. Although present for over four decades, this design procedure has had a renewed interest. The reason for this resurgence is that new manufacturing techniques (e.g., additive manufacturing) allow us to fabricate, previously unrealized, intricate material layouts often provided by \textsf{TopOpt} \citep{zegard2016bridging}.

Initially introduced for structural design, researchers have extended \textsf{TopOpt} to other fields such as electronics (e.g., MEMS) \citep{maute2003reliability,dede2015topology,jang2012topology}, automotive \citep{yang1995automotive}, biomedical \citep{fujii2002pdms,vilardell2019topology}, and aeronautical \citep{james2014concurrent, maute2004conceptual}  industries to get optimal designs for components as well as systems. These fields went beyond solid mechanics. \textsf{TopOpt} has also been applied to fluid mechanics \citep{gersborg2005topology,alexandersen2020review}, nano-photonics \citep{jensen2011topology}, acoustics \citep{duhring2008acoustic}, and flow and transport in porous media \citep{borrvall2003topology}. The last topic is of central interest to this paper. 

The emergence of ink-jet printers, which use tiny tubes carrying ink for printing, and additive manufacturing (3D printing) have triggered an interest in fluid flows through micropore networks. This interest has created fresh opportunities in the emerging research field of microfluidics \citep{whitesides2006origins}. As the name suggests, microfluidics concerns with the study of flows of tiny quantities (a few milliliters or smaller) of fluids in miniature-sized devices. In recent times, such devices have gained popularity in medical diagnostics, filters, and sensors. The success of these instruments across the application areas depends on the reliability of results and optimal performance \citep{fujii2002pdms}. Hence, the primary challenge is to get optimal designs for producing reliable results.

\citet{borrvall2003topology} were among the first to apply \textsf{TopOpt} for fluid problems. They used Stokes equations, which neglect nonlinear inertial effects, to get optimal fluid paths. Extending the mentioned work, \citet{gersborg2005topology} have included inertial effects but restricted their study to free flows. \citet{wiker2007topology} have applied \textsf{TopOpt} for coupled flows -- combining flow through porous media and free fluid flow. For the former class of flows, they used the Darcy model while the Stokes model for the latter; both these models neglect inertial effects. To drive the design, they used \emph{minimization of the total potential power} as the objective function. Another relevant work is by \citet{guest2006topology}, who proposed a mixed formulation combining the Stokes and Darcy equations with the objective, again, to minimize the total potential power. These prior works manifest two main drawbacks: (i) model the flow through porous media using Darcy equations, and (ii)  use the minimum power theorem (\textsf{MPT}) to drive the design problem. We elaborate these two points below.

Darcy equations are valid under a myriad of assumptions; see \citep{Rajagopal_2007,nakshatrala2011numerical}. Some of these assumptions are not valid to the working of typical microfluidic devices. Specifically, two fundamental assumptions of the Darcy model are: the viscosity of the fluid is constant, and the inertial effects are negligible. However, irrefutable experimental evidence shows that the viscosity of a fluid depends on the pressure, especially for organic liquids. \citet{Barus_AJS_1893_v45_p87} has shown that this dependence is exponential. Studies have shown that inertial forces can play an important role in determining the flow characteristics even with very low Reynolds numbers \citep{stone2004engineering}. A combination of various internal forces (like inertial and capillary forces) and external forces (e.g., pressure gradient) affect the microfluidic flow. As the size shrinks, surface forces become more significant than the gravity forces \citep{convery201930}. Importantly, these nonlinearities give rise to velocity and pressure profiles that are qualitatively and quantitatively different from that of Darcy equations \citep{nakshatrala2011numerical,chang2017modification}. Therefore, by using Darcy equations, the design process ignores these nonlinear effects, leading to non-optimal and unreliable products.

Returning to the second point, \textsf{MPT} is not a physical law but a mere restatement of Darcy equations, in the spirit of calculus of variations. A solution of Darcy equations satisfies \textsf{MPT}, and \emph{vice versa}. Importantly, nonlinear models for flow through porous media, like the ones considered in this paper, do not meet \textsf{MPT}. Thus, there is a search for an alternative candidate with a universal appeal for the objective function; we articulate in this paper that the rate of mechanical dissipation is an ideal candidate. 

Besides the above drawbacks, there are other aspects contributing to the knowledge gap. Even without nonlinearities, there is no general way to formulate a well-posed design problem. Should we maximize or minimize the rate of dissipation? Should we place a volume bound constraint on the material with high or low permeability? Do the boundary conditions---pressure-driven versus velocity-driven---affect the nature of the optimization problem? Last but not least, the lack of analytical solutions is impeding the verification of numerical solutions and the essential understanding of optimal designs required for translation to other scenarios.

The central aim of this paper is to fill these aspects of the knowledge gap. We will account for two nonlinearities---pressure-dependent viscosity, and nonlinear inertial effects---into the primal analysis. We will answer the following key question, pertaining to the formulation of well-posed design problems under \textsf{TopOpt} for flow through porous media applications: 
\begin{enumerate}[(Q1)]
\item What are the proper objective function and volume constraint to drive the topology optimization for problems involving flow through porous media? 
\begin{enumerate}[(i)]
\item Specifically, we will discuss the limited scope of \textsf{MPT}.
\item We show that extremization of the rate of dissipation, a physical quantity, can serve as a proper objective function. Whether it is maximization or minimization depends on the boundary conditions and the nature of the problem.
\item A volume constraint on the material with a higher permeability is compatible with a wide variety of problems. 
\end{enumerate}
\end{enumerate}

To facilitate verification of numerical simulators, we will present analytical solutions for 1D and axisymmetric problems. We will also present numerical solutions for various representative problems. Using these analytical and numerical solutions, we will gain an understanding of the optimal designs and answer these fundamental questions: 
\begin{enumerate}[(Q2)]
\item What are the ramifications of the two nonlinearities (pervasive in microfluidic applications)---pressure-dependent viscosity and inertial effects---individually on the optimal designs?
\item How do the volumetric constraint, and (pressure-driven \emph{versus} velocity-driven) boundary conditions affect the optimal material layout?
\end{enumerate}

A few words on the scope of this paper are warranted. First, this paper does not concern with solution procedures, such as new filters or solvers, which dominate the current literature. Instead, constructing well-posed design problems and understanding the resulting optimal designs are the main themes of this paper. The innovation is that our approach is very general, applicable even to nonlinear models. This general approach along with the body of knowledge presented in this paper will be valuable to researchers and practitioners alike. The design problem under the proposed \textsf{TopOpt} framework is programmed using \citet{COMSOL}, and all the numerical results reported in this paper are generated using this computer implementation.

Second, this paper addresses the flow of \emph{single-phase} fluids in \emph{rigid} porous media. However, numerous applications exist that involve the flow of multi-phase fluid systems \citep{zhao2019comprehensive}. Also, some emerging applications derive their functionalities from the porous solid's deformation; a few notable works include \citep{kim2013rigorous,dalbe2018morphodynamics}. These two phenomena---the flow of multi-phase fluids and deformable porous media---are beyond the scope of this paper. Applying topology optimization to such phenomena has not been explored adequately and can be an excellent topic for future works.

The outline for the rest of the paper is as follows. \S\ref{Sec:S2_Porous_GE} presents the equations governing fluid flow in porous media (i.e., Darcy equations, pressure-dependent viscosity, and inertial effects). In \S\ref{Sec:S3_Porous_MPT}, we will outline the minimum power theorem and discuss the appropriateness and limitations of using this theorem as an objective function to drive \textsf{TopOpt}. Using one-dimensional problems, we will then gain insights into how to construct well-posed design problems based on the rate of mechanical dissipation (\S\ref{Sec:S4_porous_1D_insights}), and construct analytical solutions for these design problems (\S\ref{Sec:S5_porous_1D_optimal_solutions}). \S\ref{Sec:S6_Porous_AxiSym_2D} and \S\ref{Sec:S7_Porous_AxiSym_3D} will consider axisymmetry problems in 2D and 3D and address two aspects: (i) how to get accurate designs respecting the underlying symmetry, and (ii) what is the nature of these designs. \S\ref{Sec:S8_Porous_Barus} discusses the ramifications of pressure-dependent viscosity and inertial effects on the optimal designs. The article will end with concluding remarks on using \textsf{TopOpt} for designing devices involving the flow of fluids through porous media (\S\ref{Sec:S9_Porous_CR}).

\section{GOVERNING EQUATIONS: MATHEMATICAL MODELS AND DESIGN PROBLEM}
\label{Sec:S2_Porous_GE}

\subsection{Mathematical models for flow through porous media}
We will first document Darcy equations and then present the nonlinear generalizations considered in this paper. We denote the domain by $\Omega$ and its boundary $\partial \Omega$. A spatial point is denoted by $\mathbf{x} \in \overline{\Omega}$, where an overline denotes the set closure (i.e., $\overline{\Omega} := \Omega \cup \partial \Omega$), the unit outward normal to the boundary by $\widehat{\mathbf{n}}(\mathbf{x})$, and the spatial gradient and divergence operators by $\mathrm{grad}[\cdot]$ and $\mathrm{div}[\cdot]$, respectively. $\mathbf{v}(\mathbf{x})$ and $p(\mathbf{x})$ denote the discharge (or Darcy) velocity and pressure\footnote{For the Darcy or Darcy-Forchheimer model, $p$ can be either absolute or relative pressure. But for the Barus model, $p$ should be interpreted as the relative pressure with an appropriate selection of the reference pressure. See equation \eqref{Eqn:Barus_Viscosity} and the discussion below it.}, respectively. The boundary is divided into two complementary parts: $\Gamma^{v}$ and $\Gamma^{p}$. $\Gamma^{v}$ is that part of the boundary on which the normal component of the velocity is prescribed, and $\Gamma^{p}$ is the part of the boundary on which the pressure is prescribed. 

The equations that govern the flow of fluids in porous media take the following form: 
\begin{subequations}
  \begin{alignat}{2}
  \label{Eqn:TopOpt_GE_Darcy}
    &\alpha \mathbf{v}(\mathbf{x}) + \mathrm{grad}[p] = \rho \mathbf{b}(\mathbf{x}) && \quad \mathrm{in} \; \Omega \\
    \label{Eqn:TopOpt_GE_Continuity}
    &\mathrm{div}[\mathbf{v}] = 0 &&\quad \mathrm{in} \; \Omega \\
    \label{Eqn:TopOpt_GE_vBC}
    &\mathbf{v}(\mathbf{x}) \cdot \widehat{\mathbf{n}}(\mathbf{x}) = v_{n}(\mathbf{x}) &&\quad \mathrm{in} \; \Gamma^{v} \\
    \label{Eqn:TopOpt_GE_pBC}
    &p(\mathbf{x}) = p_0(\mathbf{x}) &&\quad \mathrm{on} \; \Gamma^{p}
  \end{alignat}
\end{subequations}
where $\alpha$ is the drag coefficient, $\rho$ the density of the fluid, $p_0(\mathbf{x})$ the prescribed pressure, $v_n(\mathbf{x})$ the prescribed normal component of the velocity, and $\mathbf{b}(\mathbf{x})$ the given specific body force. The drag coefficient is defined as:
\begin{align}
  \alpha = \frac{\mu}{k(\mathbf{x})}
\end{align}
where $\mu$ is the coefficient of viscosity of the fluid, and $k$ is the permeability of the porous medium. The Darcy model assumes constant coefficient of viscosity and the permeability to be independent of the field variables (i.e., velocity and pressure), while the latter can vary spatially. 

\subsubsection{Pressure-dependence viscosity} Based on extensive experiments, \citet{Barus_AJS_1893_v45_p87} has shown that the coefficient of viscosity for many organic liquids depends exponentially on the pressure. Mathematically,
\begin{align}
\label{Eqn:Barus_Viscosity}
  \mu(p) = \mu_0 \exp[\beta_{\mathrm{B}} p]
\end{align}
where $\beta_\mathrm{B} \geq 0$ is a characteristic of the fluid, which can be determined experimentally, and $\mu_0$, a constant, is the viscosity of the fluid at a reference pressure. Often the reference pressure is taken to be the atmospheric pressure. A non-zero value for the reference pressure implies that $p$ should be interpreted as the relative pressure (i.e., difference between the absolute pressure and the reference pressure). Given equation \eqref{Eqn:Barus_Viscosity}, the drag coefficient takes the following mathematical form: 
\begin{align}
	\label{Eqn:Barus_Drag}
  \alpha = \alpha_{\mathrm{B}}(p(\mathbf{x}),\mathbf{x}) = \frac{\mu(p)}{k(\mathbf{x})} = \frac{\mu_0}{k(\mathbf{x})}  \exp[ \beta_\mathrm{B} p]
\end{align}
A linearized approximation, referred to as the linearized Barus model in this paper, takes the following form:
\begin{align}
	\label{Eqn:Barus_Linear}
  \alpha = \alpha_{\mathrm{LB}}(p(\mathbf{x}),\mathbf{x}) = \frac{\mu_0}{k(\mathbf{x})} [1 + \beta_\mathrm{B} p] 
\end{align}

\subsubsection{Inertial effects}
Philipp Forchheimer, an Austrian scientist, proposed a model---often referred to as Darcy-Forchheimer model---to account for the inertial effects \citep{de2012theory}. We follow \citet{chang2017modification} and express this model in the following mathematical form: 
\begin{align}
	\label{Eqn:Forchheimer_Viscosity}
  	\alpha = \alpha_{\mathrm{DF}} (\mathbf{v}(\mathbf{x}),\mathbf{x}) = \frac{\mu}{k(\mathbf{x})}
  	\left(1 +  \beta_\mathrm{F} \|\mathbf{v}\| \right)
\end{align}
where $\beta_\mathrm{F}$ is the Forchheimer coefficient, and $\|\cdot\|$ denotes the Euclidean 2-norm.

The rate of dissipation---a physical quantity with a strong mechanics and thermodynamics underpinning---will play a vital role in this paper. On this account, the total rate of dissipation is defined as:
\begin{align}
  \label{Eqn:TO_porous_total_rate_of_dissipation}
  \Phi = \int_{\Omega} \varphi(\mathbf{x}) \mathrm{d} \Omega
\end{align}
where $\varphi(\mathbf{x})$ is the rate of dissipation density. The mathematical form for $\varphi(\mathbf{x})$ (valid for all the models mentioned above---Darcy, Barus, linearized Barus and Darcy-Forchheimer) reads: 
\begin{align}
  \label{Eqn:TO_porous_dissipation_density_general}
  \varphi(\mathbf{x}) := \alpha \, \mathbf{v}(\mathbf{x}) \cdot \mathbf{v}(\mathbf{x})
\end{align}

\subsection{Design problem under \textsf{TopOpt}} Consider two porous materials with different permeabilities. The design question then is \emph{how to place these two porous materials within the domain to achieve the desired goal}. We often place a volume constraint bound on one of the materials due to cost considerations or scarcity of resources; for example, the total volume of the expensive porous material should not be more than, say, 30\% of the entire volume of the domain. We denote the design variable as $\xi(\mathbf{x})$ with the following representation:
\begin{align}
	\xi(\mathbf{x}) = \left\{\begin{array}{ll}
	1 & \text{if} \; \mathbf{x} \; \text{is occupied by the material with volume constraint }\\
	0 & \text{if} \; \mathbf{x} \; \text{is occupied by the material without volume constraint}
	\end{array} \right. \notag
\end{align}

In this paper, we formulate the design problem using the rate of dissipation for the objective function. The design problem can be either minimization or maximization of the dissipation functional. Under minimization, the proposed design problem using \textsf{TopOpt} takes the following form\footnote{The corresponding design problem under maximization can be obtained by replacing argmin with argmax.}: 
\begin{subequations}
  \label{Eqn:TopOpt_integer_programming_original}
\begin{alignat}{2}
	&\widehat{\xi}(\mathbf{x}) \leftarrow \mathop{\mathrm{argmin}}_{\xi(\mathbf{x})} \; \Phi [\xi(\mathbf{x}), \mathbf{v}(\mathbf{x}),p(\mathbf{x})] 
	&& \quad \mbox{(objective functional)} \\  \notag
	&\mbox{subject to:} \\
  \label{Eqn:TopOpt_integer_programming_original_primal}
	& \left.
	\begin{array}{ll}
	\alpha \mathbf{v}(\mathbf{x}) + \mathrm{grad}[p] = \rho \mathbf{b}(\mathbf{x}) &\; \mathrm{in} \; \Omega \\  
	\mathrm{div}[\mathbf{v}] = 0 &\; \mathrm{in} \; \Omega \\  
	\mathbf{v}(\mathbf{x}) \cdot \widehat{\mathbf{n}}(\mathbf{x}) = v_n(\mathbf{x})  &\; \mathrm{on} \; \Gamma^v \\  
	p(\mathbf{x}) = p_0(\mathbf{x}) & \; \mathrm{on} \; \Gamma^p  
	\end{array} \right\} 
	&& \quad \mbox{(state equations)} \\
	&\int_{\Omega} \xi(\mathbf{x}) \, \mathrm{d} \Omega \leq \gamma \int_{\Omega} 
	\mathrm{d} \Omega = \gamma \, \mbox{meas}(\Omega)
	&&\quad \mbox{(volume constraint)} \\
	&\xi(\mathbf{x})\in \left\{0,1\right\} \quad \forall \mathbf{x} \in \Omega
	&& \quad \mbox{(design set/space)}
\end{alignat}
\end{subequations}
where $0 \leq \gamma \leq 1$ is the bound for the volume constraint, and $\mbox{meas}(\Omega)$ denotes the area/volume of the domain.

In its primitive form, the above optimization problem is an integer programming; $\xi$ takes either 0 or 1. Such integer programming optimization problems are notoriously difficult to solve. Hence, a regularization procedure, such as SIMP \citep{bendsoe_sigmund_2013topology,bendsoe1988generating,rozvany2001aims} and MMA \citep{svanberg1987method}, is often used to convert the above problem \eqref{Eqn:TopOpt_integer_programming_original} into a nonlinear programming problem. The design space will then be the entire interval between 0 and 1 (i.e., $\xi \in [0,1]$). However, an effective regularization technique will force the design variables towards 0 or 1. The corresponding nonlinear programming problem, after regularization, takes the following form:
\begin{align}
  \mathop{\mathrm{max}}_{\xi \; \in \; \left[0,1 \right]}
  \quad \widetilde{\Phi}_\xi [\xi(\mathbf{x}),
    \mathbf{v}(\mathbf{x}),p(\mathbf{x})]
\end{align}
where $\widetilde{\Phi}_{\xi}$ is a regularized functional.

  The design problem has two main ingredients: the primal analysis and the optimization procedure. In the literature, primal analysis also goes by several other names: \emph{forward problem/model} \citep{tarantola2005inverse}, \emph{direct problem} in imaging \citep{bertero2020introduction}, \emph{solution to equilibrium equations} (especially in solid mechanics \citep{bendsoe_sigmund_2013topology}, or sometimes just plain \emph{analysis}. The central piece of primal analysis is a solution procedure that solves the state equations \eqref{Eqn:TopOpt_integer_programming_original_primal} for a fixed value of the design (field) variable $\xi(\mathbf{x})$. This solution procedure can be either analytical or numerical. The optimization procedure extremizes the objective function by satisfying the constraints as well as the state equations. The constraints can be bounds on the usage of each material, minimum length scale on the microstructure, manufacturing constraints to facilitate fabrication, to name a couple. In numerical simulations of complicated problems, both the ingredients---primal analysis and optimization procedure---are solved numerically. The selection and accuracy of numerical schemes for each of these ingredients will affect the final design outcome.

We have used the SIMP regularization and MMA algorithm in all our numerical results reported in this paper. The MMA algorithm used with SIMP (Solid Isotropic Material with Penalization) regularization method ensures that the solution steers towards a 0--1 solution where $1$ indicates the presence of the controlled material while $0$ indicates the presence of uncontrolled material. The stopping criteria---for terminating the design procedure---is to satisfy any one of the following conditions: (i) Closeness to a 0--1 solution with no apparent change in the material distribution over several consequent iterations, and the maximum number of iterations reaches 100. (ii) The relative tolerance of objective function in successive iterations less than or equal to 0.001.

\section{ON THE USE OF MINIMUM POWER THEOREM}
\label{Sec:S3_Porous_MPT}
In the literature, the minimum power theorem (\textsf{MPT}) also goes by the name: \emph{the principle of minimum power}. A few remarks about \textsf{MPT} are in order:
\begin{enumerate}
\item It is not a physical law, although its name might suggest some relation to the first law of thermodynamics. To the contrary, it is a mathematical result from the calculus of variations; a solution to the Darcy equations \eqref{Eqn:TopOpt_GE_Darcy}--\eqref{Eqn:TopOpt_GE_pBC} satisfies \textsf{MPT}, and \emph{vice versa}.
\item It is not a general principle governing the flow of fluids through porous media. Notably, \textsf{MPT} is not valid for Barus and Darcy-Forchheimer models, considered in this paper.
\end{enumerate}
The plan for the rest of this section is: we will first precisely state \textsf{MPT}. This will be followed by a discussion on how to pose a design problem, under \text{TopOpt}, based on \textsf{MPT} (but restricted to Darcy equations). Finally, we prove that \textsf{MPT} is not valid for nonlinear models. 

We start by defining the space of kinematically admissible vector fields:
\begin{align}
  \mathcal{K} := \{\widetilde{\mathbf{v}}(\mathbf{x}) \; | \;
  \mathrm{div} [\widetilde{\mathbf{v}}] = 0  \; \mathrm{in} \; \Omega, \;
  \widetilde{\mathbf{v}}(\mathbf{x}) \cdot \widehat{\mathbf{n}}(\mathbf{x}) = v_n(\mathbf{x})
  \; \mathrm{on} \; \Gamma^v  \}
\end{align}
  That is, a kinematically admissible vector field $\widetilde{\mathbf{v}}(\mathbf{x})
  \in \mathcal{K}$ needs to satisfy the continuity condition \eqref{Eqn:TopOpt_GE_Continuity}
  and the velocity boundary conditions \eqref{Eqn:TopOpt_GE_vBC}; but it is \emph{not}
  necessary for a kinematically admissible vector field to satisfy either the balance
  of linear momentum \eqref{Eqn:TopOpt_GE_Darcy} or the pressure boundary condition
  \eqref{Eqn:TopOpt_GE_pBC}.\footnote{A kinematically admissible vector field
    is similar to a virtual displacement or a virtual velocity, which arise in the fields
    such as particle and rigid body dynamics as well as continuum mechanics. A virtual
    velocity, in continuum mechanics, satisfies the velocity boundary conditions and
    internal constraints (e.g., the incompressible constraint, similar to equation
    \eqref{Eqn:TopOpt_GE_Continuity}), but not the balance of linear momentum.
    For a more detailed exposure, see \citep{rosenberg1977analytical,germain1973method}.}
We call a vector field $\mathbf{v} : \overline{\Omega} \rightarrow \mathbb{R}^{nd}$ the Darcy velocity, a special type of kinematically admissible field, if it satisfies equations \eqref{Eqn:TopOpt_GE_Darcy}--\eqref{Eqn:TopOpt_GE_pBC}. That is, $\mathbf{v}(\mathbf{x}) \in \mathcal{K}$ and satisfies equations \eqref{Eqn:TopOpt_GE_Darcy} and \eqref{Eqn:TopOpt_GE_pBC}.

The central piece of \textsf{MPT} is the mechanical power functional, defined as follows:
\begin{align}
	\label{Eqn:PMP}
	\Psi [\mathbf{w}(\mathbf{x})]
	:= \int_{\Omega} \frac{1}{2} \; \alpha \; \mathbf{w}(\mathbf{x}) \cdot \mathbf{w}(\mathbf{x}) \; \mathrm{d}\Omega
	-\int_{\Omega} \mathbf{w}(\mathbf{x}) \cdot \rho \mathbf{b}(\mathbf{x}) \; \mathrm{d}\Omega
	+\int_{\Gamma^p} \mathbf{w}(\mathbf{x}) \cdot \widehat{\mathbf{n}}(\mathbf{x}) \; p_0(\mathbf{x}) \; \mathrm{d}\Gamma
\end{align}
where $\mathbf{w(x)} \in \mathcal{K}.$

\textsf{MPT} can then be compactly written as:
\begin{align}
	\Psi [\mathbf{v}(\mathbf{x})] \; \leq \; \Psi [\widetilde{\mathbf{v}}(\mathbf{x})] \quad \forall  \widetilde{\mathbf{v}}(\mathbf{x}) \in \mathcal{K}
\end{align}
In plain words, \textsf{MPT} states that the Darcy velocity renders the minimum mechanical power among all the kinematically admissible vector fields.  Using the language of optimization theory, an alternate form of \textsf{MPT} is:
\begin{align}
	\mathbf{v}(\mathbf{x}) \leftarrow
	\mathop{\mathrm{argmin}}_{\widetilde{\mathbf{v}}(\mathbf{x}) \in \mathcal{K} } \; \Psi[\widetilde{\mathbf{v}}(\mathbf{x})]
\end{align}

A design problem, under \textsf{TopOpt}, based on MPT can be posed as follows: 
\begin{align}
\label{Eqn:TopOpt_integer_programming}
	\mathop{\mathrm{min}}_{\xi \; \in \; \left\{0,1 \right\} } \quad \mathop{\mathrm{min}}_{\widetilde{\mathbf{v}}(\mathbf{x}) \; \in \; \mathcal{K}} \quad \Psi[\widetilde{\mathbf{v}}(\mathbf{x})]
\end{align}
where $\xi$ is the design variable. The outer optimization problem is an integer programming (i.e., $\xi$ takes either 0 or 1). The corresponding nonlinear programming problem after regularization can be written as follows:
\begin{align}
	\mathop{\mathrm{min}}_{\xi \; \in \; \left[0,1 \right] } \quad \mathop{\mathrm{min}}_{\widetilde{\mathbf{v}}(\mathbf{x}) \; \in \; \mathcal{K}} \quad \widetilde{\Psi}_\xi [\widetilde{\mathbf{v}}(\mathbf{x})]
\end{align}
where $\widetilde{\Psi}_{\xi}$ is a regularized functional.

\citet{shabouei_nakshatrala_cicp} have provided a proof of \textsf{MPT} for Darcy equations. We now show that \textsf{MPT} will not hold for nonlinear models such as Darcy-Forchheimer and Barus. Thus, in our derivation, we allow the drag coefficient $\alpha$ to depend on the solution fields---$\mathbf{v}(\mathbf{x})$ and $p(\mathbf{x})$. To handle the incompressibility constraint, it is convenient to use the Lagrange multiplier method. On this account, we define the following:
\begin{align}
\mathcal{V} &:= \left\{\mathbf{v}(\mathbf{x}) \, | \, \mathbf{v}(\mathbf{x}) \cdot \widehat{\mathbf{n}}(\mathbf{x}) = v_{n}(\mathbf{x}) \; \mathrm{on} \; \Gamma^{v}\right\} \\
\mathcal{V}_{0} &:= \left\{\mathbf{v}(\mathbf{x}) \, | \, \mathbf{v}(\mathbf{x}) \cdot \widehat{\mathbf{n}}(\mathbf{x}) = 0 \; \mathrm{on} \; \Gamma^{v}\right\}
\end{align}
Appealing to the Lagrange multiplier method, we consider the following equivalent optimization problem: 
\begin{align}
\mathop{\mathrm{extremize}}_{\mathbf{v}(\mathbf{x}) \in \mathcal{V},\, p(\mathbf{x})} \quad \widehat{\Psi}[\mathbf{v}(\mathbf{x}),p(\mathbf{x})]
\end{align}
where $p(\mathbf{x})$, in addition to being the (mechanical) pressure, serves as the Lagrange multiplier enforcing the incompressibility constraint, and 
\begin{align}
  \widehat{\Psi}[\mathbf{v}(\mathbf{x}),p(\mathbf{x})] 
  &:= \Psi[\mathbf{v}(\mathbf{x})]   
  - \int_{\Omega} p(\mathbf{x}) \mathrm{div}[\mathbf{v}] \, \mathrm{d} \Omega 
\end{align}

For the solution fields to satisfy \textsf{MPT}, the necessary condition is:
\begin{align}
\delta \widehat{\Psi}[\mathbf{v}(\mathbf{x}),p(\mathbf{x})] \cdot \left(\delta \mathbf{v}(\mathbf{x}), \delta p(\mathbf{x})\right) := \left[\frac{d}{d \epsilon} \widehat{\Psi}[\mathbf{v}(\mathbf{x}) + \epsilon \, \delta \mathbf{v}(\mathbf{x}),p(\mathbf{x})+\epsilon \, \delta p(\mathbf{x})]\right]_{\epsilon =0} = 0 \notag \\ 
\quad \forall \delta \mathbf{v}(\mathbf{x}) \in  \mathcal{V}_0, \forall \delta p(\mathbf{x}) 
\end{align}
Thus, it suffices to show that this necessary condition is not met if $\alpha$ depends on the solution fields. We proceed as follows:
\begin{align}
\delta \widehat{\Psi}[\mathbf{v}(\mathbf{x}),p(\mathbf{x})] \cdot \left(\delta \mathbf{v}(\mathbf{x}),\delta p(\mathbf{x})\right) &=
\int_{\Omega} \alpha \, \mathbf{v}(\mathbf{x}) \cdot \delta \mathbf{v}(\mathbf{x}) \, \mathrm{d} \Omega
- \int_{\Omega} \rho \mathbf{b}(\mathbf{x}) \cdot \delta \mathbf{v}(\mathbf{x}) \, \mathrm{d} \Omega
+ \int_{\Gamma^{p}} p_0(\mathbf{x}) \, \delta \mathbf{v}(\mathbf{x}) \cdot \widehat{\mathbf{n}}(\mathbf{x}) \, \mathrm{d} \Gamma \notag \\
&- \int_{\Omega} p(\mathbf{x}) \, \mathrm{div}[\delta \mathbf{v}(\mathbf{x})] \, \mathrm{d} \Omega 
- \int_{\Omega} \delta p(\mathbf{x}) \, \mathrm{div}[\mathbf{v}(\mathbf{x})] \, \mathrm{d} \Omega
\notag \\ 
&+ \int_{\Omega} \frac{1}{2} \left(\frac{\partial \alpha}{\partial \mathbf{v}} \cdot \delta \mathbf{v}
+ \frac{\partial \alpha}{\partial p} \cdot \delta p
\right)
\mathbf{v}(\mathbf{x}) \cdot \mathbf{v}(\mathbf{x}) \mathrm{d} \Omega
\end{align}
Using the Green's identity, noting that the solution fields satisfy equations \eqref{Eqn:TopOpt_GE_Darcy}--\eqref{Eqn:TopOpt_GE_pBC}, and invoking $\delta \mathbf{v}(\mathbf{x}) \in \mathcal{V}_0$, we end up with:
\begin{align}
  \delta \widehat{\Psi}[\mathbf{v}(\mathbf{x}),p(\mathbf{x})] \cdot (\delta \mathbf{v}(\mathbf{x}),\delta p(\mathbf{x})) 
  &= \int_{\Omega} \frac{1}{2} \left(\frac{\partial \alpha}{\partial \mathbf{v}} \cdot \delta \mathbf{v}
+ \frac{\partial \alpha}{\partial p} \cdot \delta p
\right)
\mathbf{v}(\mathbf{x}) \cdot \mathbf{v}(\mathbf{x}) \mathrm{d} \Omega
\end{align}
For the above integral to vanish for all $\delta \mathbf{v}(\mathbf{x}) \in \mathcal{V}_0$ and for all $p(\mathbf{x})$, we must have: 
\begin{align}
  \frac{\partial \alpha}{\partial \mathbf{v}} = \mathbf{0} \quad \mathrm{and} \quad \frac{\partial \alpha}{\partial p} = 0
\end{align}
meaning that the drag function is \emph{in}dependent of the solution fields.

Since \textsf{MPT} is limited in scope, we need a general way to pose design problems under \textsf{TopOpt} for flow through porous media applications. All the studies from hereon uses the rate of dissipation functional.

\section{INSIGHTS FROM ONE-DIMENSIONAL PROBLEMS}
\label{Sec:S4_porous_1D_insights}
We will use the Darcy model and one-dimensional
problems to gain insights into three aspects: 
\begin{enumerate}
\item \emph{nature of the solution:} How does
  the rate of dissipation density depend on
  permeability $k$?
\item \emph{extremization of the objective
  function:} Is it physical to minimize or
  maximize $\Phi$?
\item \emph{volumetric bound constraint:} On
  which material---high or low permeability---should
  we place a bound on its volume fraction?
\end{enumerate}

We start by noting the rate of dissipation
density under the Darcy model is: 
\begin{align}
  \label{Eqn:TO_porous_dissipation_density_Darcy}
  \varphi(\mathbf{x}) = \frac{\mu}{k}
  \mathbf{v}(\mathbf{x}) \cdot \mathbf{v}(\mathbf{x}) 
\end{align}
A cavalier look at equation \eqref{Eqn:TO_porous_dissipation_density_Darcy} might suggest answer to the first question is trivial---$\varphi$ is inversely proportional to $k$. This conclusion seems even more plausible by noting that the continuity equation (i.e., balance of mass considering the incompressibility of the fluid) in 1D is:
\begin{align}
  \label{Eqn1D_Continuity}
  \mathrm{div}[\mathbf{v}] = \frac{dv}{dx} = 0 
\end{align}
which implies that the velocity $v(x) =
\mathrm{constant}$. But this reasoning
does not consider the possibility that
the velocity, although a constant for 1D
problems, could depend on $k$. That is,
the rate of dissipation density depends
on $k$ explicitly while it could also
depend on $k$ implicitly via the velocity
field.
To prove this point and to establish the exact
dependence of $\varphi$ on $k$, we will consider
two types of boundary value problems: pressure-driven
and velocity-driven. In a pressure-driven problem,
pressure boundary conditions are prescribed on the
entire boundary. While for a velocity-driven problem,
velocity boundary conditions are prescribed at least
on a part of the boundary.

\subsection{Pressure-driven problem}
\label{Subsec:TopOpt_pressure_driven_1D}
Consider a one-dimensional problem with length
$L$ as detailed in figure \ref{Fig:1D_Pressure}.
The pressure conditions are applied on the either
ends of the domain with $p_L > p_R$.
\begin{figure}[h]
  \includegraphics[scale=0.75]{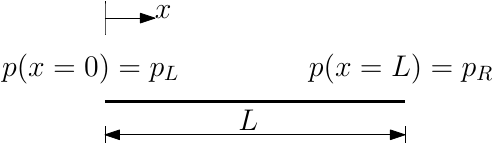}
  \caption{A pictorial description of the one-dimensional pressure-driven problem. \label{Fig:1D_Pressure}}
\end{figure}
The equations governing the flow of the fluid in the domain are:
\begin{subequations}
  \begin{alignat}{2}
    \label{Eqn:1D_porous_Darcy}
    &\frac{\mu}{k} v(x) + \frac{d p}{d x} = 0
    &&\quad \mathrm{in} \; (0, L) \\
    \label{Eqn:1D_porous_Continuity}
    &\frac{d v}{d x} = 0
    &&\quad \mathrm{in} \; (0, L) \\
    \label{Eqn:1D_porous_BCL}
    &p(x = 0) = p_L \\
    \label{Eqn:1D_porous_BCR}
    &p(x = L) = p_R
  \end{alignat}
\end{subequations}
The corresponding analytical solution is:
\begin{align}
  p(x) = \left(1-\frac{x}{L}\right) p_L + \frac{x}{L} p_R 
  \quad \mathrm{and} \quad
  v(x) &= \frac{k}{\mu} \frac{(p_L - p_R)}{L}
\end{align}
Although the velocity is a constant, as mentioned
earlier, it depends on the permeability. Accordingly,
the rate of dissipation density will be:
\begin{align}
  \label{Eqn:TO_porous_1D_Phi}
  \varphi(x) &= \frac{\mu}{k} \mathbf{v(x) \cdot v(x)}  
  = \frac{\mu}{k} v^{2}(x) 
  = \frac{k}{\mu} \frac{\left(p_L - p_R \right)^2}{L^2}
\end{align}
Thus, for the chosen pressure-driven problem,
$\varphi(x)$ is proportional to $k$ (and not
to $k^{-1}$). 

\subsubsection{A design problem for pressure-driven problems}
In the applications of pressure-driven problems, the goal often is to dissipate the (kinetic) energy as much as possible as the fluid passes through the porous domain. Thus, maximizing the total rate of dissipation is a physically meaningful choice for the objective function under \textsf{TopOpt}.

To decide on the volumetric bound, we need to assess the combined effect of this objective function alongside how the rate of dissipation density depends on the permeability for pressure-driven problems. Since $\varphi \propto k$, one needs to place a bound on the use of high-permeability material in the domain. Otherwise, we will end up with a trivial solution---placing the high-permeability material throughout the domain. Other reasons such as practical and economic could also compel to place a restriction on the use of high-permeability. On the practical front, a high-permeability material often has low strength, and too much use of such a material could comprise the structural integrity of the device. On the economic front, a high-permeability material is often an engineered material, making it expensive. All these factors justify placing a bound on the use of the high-permeability material.

Summarizing, a well-posed design problem for pressure-driven problems can be: \emph{maximization of the total rate of dissipation subject to a volume constraint bound on the high-permeability material}. 

\subsection{Velocity-driven problem} We now consider
a one-dimensional problem similar to the one used in
\S\ref{Subsec:TopOpt_pressure_driven_1D}; but the left
end of the domain is prescribed with a velocity boundary
condition (see figure \ref{Fig:1D_Vel-Pressure}).
\begin{figure}[h]
  \includegraphics[scale=0.75]{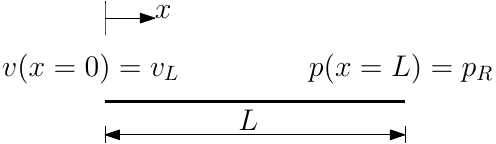}
  \caption{A pictorial description of the
    one-dimensional velocity-driven problem.
    \label{Fig:1D_Vel-Pressure}}
\end{figure}

The governing equations remain the same
as the previous case expect equation
\eqref{Eqn:1D_porous_BCL} is now replaced
with: 
\begin{align}
  \label{Eqn:1Dvel_porous_BCL}
  &v(x=0) = v_L
\end{align}
The corresponding analytical solution for the fields is: 
\begin{align}
  v(x) = v_L
  \quad \mathrm{and} \quad 
  p(x) = p_R + \frac{\mu}{k} v_L (L-x)
\end{align}
For this case, the velocity is a constant
independent of $k$. The corresponding
rate of dissipation density is:
\begin{align}
  \varphi(x) = \frac{\mu}{k} v^2(x) = \frac{\mu}{k} v_L^2 
\end{align}
Since $\mu$, $L$ and $v_L$ are independent of
$k$, $\varphi(x)$ is constant and is inversely
proportional to $k$.

\subsubsection{A design problem for velocity-driven problems}
In the applications involving velocity-driven problems, a continuous supply of power/energy will be needed for maintaining adequate flow---required to meet the velocity boundary condition. Therefore, it is desirable to minimize the rate of dissipation for these kinds of problems.
Noting that $\varphi \propto k^{-1}$ for velocity-driven problems, the minimization of the total rate of dissipation will require a volume constraint bound on the high-permeability material. Otherwise, similar to pressure-driven problems, the solution to the design problem becomes trivial.
Summarizing, a well-posed design problem for velocity-driven problems can be: \emph{minimization of the total rate of dissipation with a volumetric bound constraint on the high-permeability material}.

Although our above arguments are based on prototypical 1D problems, the conclusions are applicable even to other problems and to the nonlinear models mentioned in \S\ref{Sec:S2_Porous_GE}. The two design problems given above can guide posing other design problems.

\section{OPTIMAL LAYOUTS FOR ONE-DIMENSIONAL PROBLEMS}
\label{Sec:S5_porous_1D_optimal_solutions}
We will analyze one-dimensional problems similar to the ones considered in \S\ref{Sec:S4_porous_1D_insights}. But we will use nonlinear modifications of the Darcy model. The design problem is to place two porous materials with different permeabilities, $k_1$ and $k_2$, in the domain. The objective function and the volumetric bound constraint are the same as discussed earlier.

Using the analytical solutions to these 1D problems, the effect of nonlinearities on the optimal material layouts will be quantified. To derive analytical solution, we assume the presence of a single material interface; its location is at $x = \xi$. That is, the spatial distribution of the permeability is:
\begin{align}
  \label{Eqn:TopOpt_1D_permeability_field}
  k(x) = \left\{\begin{array}{ll}
  k_1 & 0 < x < \xi \\
  k_2 & \xi < x < L
  \end{array} \right.
\end{align}
For convenience and for a later use, we will use $\xi^{\pm}$
to denote the one-sided limits at the interface. That is,
a superscript minus sign on $\xi$ indicates a limit from
the left-side; a similar interpretation holds for the
superscript plus sign on $\xi$.

\subsection{Pressure-driven problem}
For non-dimensionalization, we take the
following (pressure, length and viscosity)
as reference quantities:
\begin{align}
  \label{Eqn:TopOpt_1D_pressure_driven_ND_RQ}
  p_\mathrm{ref} = p_L - p_R \;
  [\mathrm{M}^{1} \mathrm{L}^{-1} \mathrm{T}^{-2}], \quad
  L_\mathrm{ref} = L \;
  [\mathrm{M}^{0}\mathrm{L}^{1}\mathrm{T}^{0}]
  \quad \mathrm{and} \quad
  \mu_\mathrm{ref} = \mu_0 \exp[\beta_{\mathrm{B}} p_{\mathrm{R}}]\;
     [\mathrm{M}^{1} \mathrm{L}^{-1} \mathrm{T}^{-1}]
\end{align}
Note that $\beta_{\mathrm{B}} = 0$ for the Darcy and
Darcy-Forchheimer models. These references quantities
give rise to the following non-dimensional quantities:
\begin{align}
  \label{Eqn:TO_porous_1D_pressure_ND}
  \bar{x} = \frac{x}{L_{\mathrm{ref}}}, \;
  \bar{k} = \frac{k}{L^2_{\mathrm{ref}}}, \;
  \bar{p} = \frac{p - p_{\mathrm{R}}}{p_{\mathrm{ref}}}, \;
  \bar{\mu} = \frac{\mu}{\mu_{\mathrm{ref}}}, \;
  \bar{v} = \frac{v \mu_{\mathrm{ref}}}{L_{\mathrm{ref}} p_{\mathrm{ref}}}, \;
  \overline{\rho b} =
  \frac{\rho b L_{\mathrm{ref}}}{p_{\mathrm{ref}}}, \notag \\
  \bar{\beta}_{\mathrm{F}} =
  \frac{\beta_{\mathrm{F}} L_{\mathrm{ref}} p_{\mathrm{ref}}}{\mu_{\mathrm{ref}}}
  \quad \mathrm{and} \quad
  \bar{\beta}_{\mathrm{B}} =
  \beta_{\mathrm{B}} p_{\mathrm{ref}}
\end{align}
Then the governing equations for the primal
analysis can be written as follows:
\begin{subequations}
  \begin{alignat}{2}
    \label{Eqn:TO_porous_1D_anal_BoLM}
    &\frac{\bar{\mu}}{\bar{k}(\bar{x})} \bar{v}
    + \frac{d \bar{p}}{d \bar{x}} = \overline{\rho b} = 0
    &&\quad \forall \bar{x} \in (0, 1) \\
    \label{Eqn:TO_porous_1D_anal_continuity}
    &\frac{d \bar{v}}{d \bar{x}} = 0
    &&\quad \forall \bar{x} \in (0, 1) \\
    \label{Eqn:TO_porous_1D_anal_pL}
    &\bar{p}(\bar{x} = 0) = 1 \\
    \label{Eqn:TO_porous_1D_anal_pR}
    &\bar{p}(\bar{x} = 1) = 0
  \end{alignat}
\end{subequations}
In addition, the jump conditions
at the material interface read:
\begin{subequations}
  \begin{align}
    &\bar{v}(\bar{x} = \bar{\xi}^{-}) = \bar{v}(\bar{x} = \bar{\xi}^{+}) \\
    &\bar{p}(\bar{x} = \bar{\xi}^{-}) = \bar{p}(\bar{x} = \bar{\xi}^{+})
  \end{align}
\end{subequations}
where $\bar{\xi}^{\pm} = \xi^{\pm} / L_{\mathrm{ref}}$.
From hereon we will drop the superposed overlines for
convenience; all the quantities in the rest of this
subsection should be considered as non-dimensional.

For convenience, we define:
\begin{align}
  \Upsilon_{\mathrm{1D}}(\xi) := \frac{\xi}{k_1} + \frac{1 - \xi}{k_2}
\end{align}
where $k_1$ and $k_2$ non-dimensional permeabilities of the two given materials (cf. equations \eqref{Eqn:TopOpt_1D_permeability_field} and $\eqref{Eqn:TO_porous_1D_pressure_ND}_2$). We will soon see that this quantity captures the variation of the total rate of dissipation with the design variable---the location of the material interface, and allows to express the
analytical expressions in a compact form.

\subsubsection{Primal analysis under the Barus model}
The balance of linear momentum
\eqref{Eqn:TO_porous_1D_anal_BoLM}
for this case reads:
\begin{align}
  &\frac{1}{k} \exp[\beta_{\mathrm{B}} p] v + \frac{dp}{dx} = 0
  \quad \forall x \in (0,1)
\end{align}
The analytical solution for the velocity field is:
\begin{align}
  v(x) &= C = \frac{1}{\beta_{\mathrm{B}}}
  \left(1 - \exp[-\beta_{\mathrm{B}}] \right)
  \Upsilon_{\mathrm{1D}}^{-1}(\xi)
\end{align}
The corresponding solution for the pressure field is:
\begin{align}
  p(x) &= \left\{\begin{array}{ll}
  -\frac{1}{\beta_{\mathrm{B}}} \ln\left[
    \frac{\beta_{\mathrm{B}} C x}{k_1} + \exp[-\beta_{\mathrm{B}}]\right]
  & \quad 0 \leq x < \xi \\ \\
    -\frac{1}{\beta_{\mathrm{B}}} \ln\left[
    1 - \frac{\beta_{\mathrm{B}} C (1 - x)}{k_2} \right]
  & \quad \xi < x \leq 1
  \end{array} \right.
\end{align}
The corresponding total rate of dissipation is:
\begin{align}
  \Phi = \int_{0}^{1} \frac{\exp[\beta_{\mathrm{B}}p]}{k(x)}
  v^{2}(x) \mathrm{d} x = C = \frac{1}{\beta_{\mathrm{B}}}
  \left(1 - \exp[-\beta_{\mathrm{B}}] \right)
  \Upsilon^{-1}_{\mathrm{1D}}(\xi)
\end{align}

\subsubsection{Primal analysis under the Darcy-Forchheimer model}
The governing equation for the balance of linear
momentum for this case reads:
\begin{align}
  \label{Eqn:TopOpt_1D_DF_pressure_driven_BoLM}
  \frac{1}{k(x)}(1 + \beta_{\mathrm{F}}|v|) v
  + \frac{dp}{dx} = 0 \quad \forall x \in (0,1)
\end{align}
There are two (mathematical) solutions for the
velocity field of which one of them is unphysical.
The physical solution for the velocity field is:
\begin{align}
  v(x) = C = \frac{\sqrt{1 + 4 \beta_{\mathrm{F}}
      \Upsilon_{\mathrm{1D}}^{-1}(\xi)} - 1}{2 \beta_{\mathrm{F}}}
\end{align}
The corresponding pressure field is:
\begin{align}
  p(x) = \left\{\begin{array}{ll}
  - \frac{1}{k_1} (1 + \beta_{\mathrm{F}} \, C)
  C x + 1
  & \quad 0 \leq x < \xi \\ \\
  - \frac{1}{k_2} (1 + \beta_{\mathrm{F}} \, C)
  C (x - 1)
  & \quad \xi < x \leq 1
  \end{array} \right.
\end{align}
Thus, the total rate of dissipation is:
\begin{align}
  \Phi(\xi) = \int_{0}^{1}
  \left(\frac{1 + \beta_{\mathrm{F}} |v|}{k(x)}\right)
  v^{2}(x) \, \mathrm{d} x
  = C = \frac{\sqrt{1 + 4 \beta_{\mathrm{F}}
      \Upsilon_{\mathrm{1D}}^{-1}(\xi)} - 1}{2 \beta_{\mathrm{F}}}
\end{align}

\subsection{Velocity-driven problem}

We take the following (velocity, length and viscosity) as the reference quantities for the non-dimensionalization:
\begin{align}
  v_\mathrm{ref} = v_{\mathrm{L}} \;
  [\mathrm{M}^{0} \mathrm{L}^{1} \mathrm{T}^{-1}], \quad
  L_\mathrm{ref} = L \;
  [\mathrm{M}^{0}\mathrm{L}^{1}\mathrm{T}^{0}]
  \quad \mathrm{and} \quad
  \mu_\mathrm{ref} = \mu_0 (1 + \beta_{\mathrm{B}} p_{\mathrm{R}}) \;
     [\mathrm{M}^{1} \mathrm{L}^{-1} \mathrm{T}^{-1}]
\end{align}
where $v_{\mathrm{L}}$ is the prescribed velocity
at the left end of the domain. Recall that
$\beta_{\mathrm{B}} = 0$ for the Darcy and
Darcy-Forchheimer models, and $\beta_{\mathrm{F}} = 0$
for the Darcy and Barus models. These reference quantities introduce the following reference pressure:
\begin{align}
  p_{\mathrm{ref}} = \frac{\mu_{\mathrm{ref}} v_{\mathrm{ref}}}{L_{\mathrm{ref}}}
  = \frac{\mu_0 (1 + \beta_{\mathrm{B}} p_{\mathrm{R}}) v_{\mathrm{L}}}{L}
\end{align}
The corresponding non-dimensional quantities will be:
\begin{align}
  \bar{v} = \frac{v}{v_{\mathrm{ref}}}, \;
  \bar{\beta}_{\mathrm{F}} =
  \beta_{\mathrm{F}} v_{\mathrm{ref}}
  \quad \mathrm{and} \quad
  \bar{\beta}_{\mathrm{B}} =
  \frac{\beta_{\mathrm{B}} p_{\mathrm{ref}}}{1 + \beta_{\mathrm{B}} p_{\mathrm{R}}}
  = \frac{\beta_{\mathrm{B}} \mu_0
    v_{\mathrm{L}}}{L}
\end{align}
while the definitions for the rest of the quantities will be the same as before (i.e., equation \eqref{Eqn:TopOpt_1D_pressure_driven_ND_RQ}) but based on the new reference quantities defined above. The governing equations in non-dimensional
form (without overlines) are:
\begin{subequations}
  \begin{alignat}{2}
    \label{Eqn:TO_porous_1D_anal_BoLM_VD}
    &\frac{\mu}{k(x)} v
    + \frac{d p}{d x} = 0
    &&\quad \forall x \in (0, 1) \\
    \label{Eqn:TO_porous_1D_anal_continuity_VD}
    &\frac{d v}{d x} = 0
    &&\quad \forall x \in (0, 1) \\
    \label{Eqn:TO_porous_1D_anal_vL_VD}
    &v(x = 0) = 1 \\
    \label{Eqn:TO_porous_1D_anal_pR_VD}
    &p(x = 1) = 0
  \end{alignat}
\end{subequations}

\subsubsection{Primal analysis under the linearized Barus model}
We have considered the linearized Barus model as for velocity-driven problems primal solutions might not even exist under the (nonlinear) Barus model. This is the case with the above boundary value problem for some values of prescribed velocity at the inlet and prescribed pressure at the outlet. A thorough discussion of this issue is beyond the scope of this paper and will be dealt elsewhere.

The non-dimensional form of the balance
of linear momentum reads:
\begin{align}
  \left(\frac{1 + \beta_{\mathrm{B}} p}{k(x)}\right) v
  + \frac{dp}{dx} = 0
  \quad \forall x \in (0,1)
\end{align}
The solution fields are:
\begin{align}
  v(x) &= 1 \\
  p(x) &= \left\{\begin{array}{ll}
  \frac{1}{\beta_{\mathrm{B}}}
  \left( \exp[\beta_{\mathrm{B}} \Upsilon_{\mathrm{1D}}]
    \exp\left[-\frac{x\beta_{\mathrm{B}}}{k_1}\right] - 1\right)
  & \quad 0 \leq x < \xi \\ \\
  \frac{1}{\beta_{\mathrm{B}}}
  \left(\exp\left[\frac{\beta_{\mathrm{B}}}{k_2}(1 - x)\right] - 1\right)
  & \quad \xi < x \leq 1
  \end{array} \right.
\end{align}
The total rate of dissipation is:
\begin{align}
  \Phi(\xi) = \frac{1}{\beta_{\mathrm{B}}}
  \left(\exp[\beta_{\mathrm{B}} \Upsilon_{\mathrm{1D}}(\xi)] - 1\right)
\end{align}

\subsubsection{Primal analysis under the Darcy-Forchheimer model}
The governing equation for the balance of linear momentum will be the same as given in equation \eqref{Eqn:TopOpt_1D_DF_pressure_driven_BoLM}. The solution fields are:
\begin{align}
  v(x) &= 1 \\
  p(x) &= \left\{\begin{array}{ll}
  (1 + \beta_{\mathrm{F}}) \Upsilon_{\mathrm{1D}}
  -\left(\frac{1 + \beta_{\mathrm{F}}}{k_1}\right) x
  & \quad 0 \leq x < \xi \\ \\
    \left(\frac{1 + \beta_{\mathrm{F}}}{k_2}\right)(1 - x) & \quad \xi < x \leq 1
  \end{array} \right.
\end{align}
The total rate of dissipation is:
\begin{align}
  \Phi(\xi) = (1 + \beta_{\mathrm{F}})
  \Upsilon_{\mathrm{1D}}(\xi)
\end{align}

\subsection{Optimal material layout}
In the above solutions, $\Phi \propto \Upsilon_{\mathrm{1D}}^{-1}(\xi)$ for pressure-driven problems and $\Phi \propto \Upsilon_{\mathrm{1D}}(\xi)$ for velocity-driven problems. Noting the total rate of dissipation is maximized for pressure-driven problem while it is minimized for velocity-driven, both these cases lead to the minimization of $\Upsilon_{\mathrm{1D}}(\xi)$ with respect to $\xi$ satisfying the volumetric bound constraint on the high-permeability material.

Under the single interface assumption, given two materials with permeabilities $k_{H} > k_L$, the design problem reduces to checking which of the two cases---placing the high-permeability material at the inflow or at the outflow---renders minimum $\Upsilon_{\mathrm{1D}}$. It so happens that in 1D, given the mathematical structure of $\Upsilon_{\mathrm{1D}}(\xi)$, high-permeability material can be placed either at the inflow or outflow as long as the volumetric constraint bound is met. In fact, it can be shown by considering multiple material interfaces that the high-permeability can be placed anywhere within
the domain (need not be contiguous), as long as it takes up the whole volume
constraint; however, this is not true in higher dimensions. Moreover, the optimal material layout is the same in all the models considered (i.e., Darcy model and its nonlinear generalizations).

\section{OPTIMAL LAYOUTS FOR A 2D AXISYMMETRIC PROBLEM}
\label{Sec:S6_Porous_AxiSym_2D}

Consider a two-dimensional domain comprising two concentric circles with an inner radius $r_i$ and an outer radius $r_o > r_i$. A pressure $p_i$ is applied on the inner boundary, and the outer boundary is subject to a pressure $p_o < p_i$; see figure \ref{Fig:Conc_cyln_BVP}. We will obtain the optimal material distribution in the domain under the Darcy model with two materials of different permeabilities: $k_1$ and $k_2$. A volume constraint, denoted by $\gamma$, is placed on the high-permeability material. We will take the maximization of rate of dissipation as the objective function.

\begin{figure}[h]
  \includegraphics[scale=0.5]{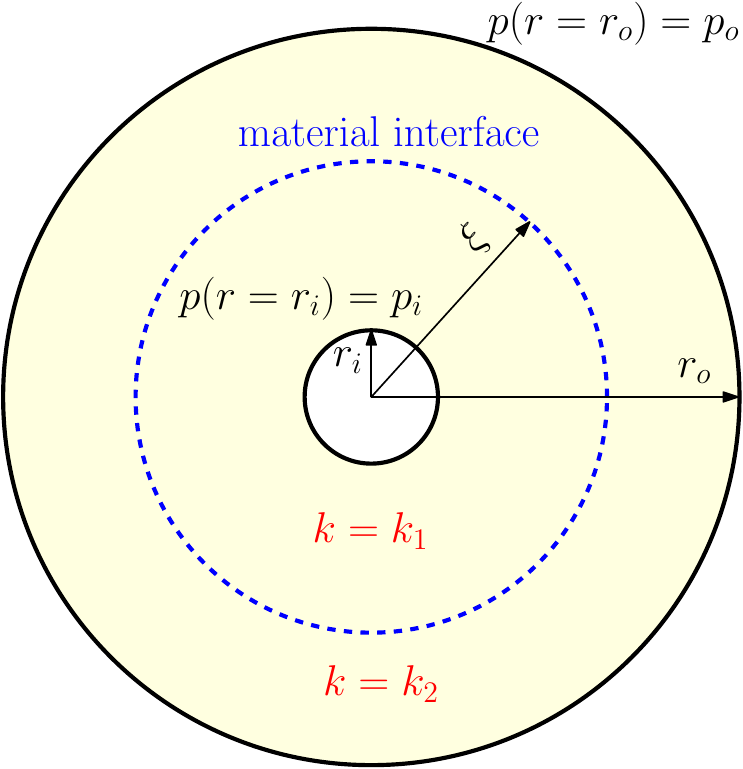}
  \caption{A pictorial description of the boundary value problem. The inner boundary is subject to a pressure $p = p_i$, and the outer boundary is subject to pressure $p = p_o$. \label{Fig:Conc_cyln_BVP}}
\end{figure}

\subsection{Analytical solution for the optimal design}
In deriving the analytical solution, we assume that each material is present in a contiguous and symmetric manner, and a single boundary exists between the two materials over the entire domain (cf. material interface shown in figure \ref{Fig:Conc_cyln_BVP}). However, such material layout might not furnish us with the optimal solution. But we will show later that numerical simulations confirm the validity of our assumption.

We will use cylindrical polar coordinates to take the advantage of axisymmetry. Then the assumption of a single material interface allows us to write the spatial dependence of the permeability as follows:
\begin{align}
	k(r) = \left\{\begin{array}{ll}
		k_1 & r_i < r < \xi \\
		k_2 & \xi < r < r_o
	\end{array} \right.
\end{align}
The governing equations corresponding to the primal analysis take the following form:
\begin{subequations}
  \begin{alignat}{2}
    \label{Eqn:TO_porous_Cylinder_Darcy}
    &\frac{\mu}{k(r)} v_r + \frac{d p}{d r} = 0
    &&\quad \mathrm{in} \; (r_i, r_o) \\
    \label{Eqn:TO_porous_Cylinder_Continuity}
    &\frac{1}{r} \frac{d (r v_r)}{d r} = 0
    &&\quad \mathrm{in} \; (r_i,r_o) \\
    \label{Eqn:TO_porous_Cylinder_BCi}
    &p(r = r_i) = p_i \\
    \label{Eqn:TO_porous_Cylinder_BCo}
    &p(r = r_o) = p_o
  \end{alignat}
\end{subequations}
where $v_r(r)$ denotes the radial component of the
velocity. In addition, the jump conditions at the
material interface take the following form:
\begin{subequations}
  \begin{align}
    \label{Eqn:TO_porous_Cylinder_jump_pressure}
    &p(r = \xi^{-}) = p(r = \xi^{+}) \\
        \label{Eqn:TO_porous_Cylinder_jump_velocity}
    &v_{r}(r = \xi^{-}) = v_r(r = \xi^{+})
  \end{align}
\end{subequations}
where $\xi^{\pm}$ denotes the one-sided limits.

Equations \eqref{Eqn:TO_porous_Cylinder_Continuity}
and \eqref{Eqn:TO_porous_Cylinder_jump_velocity} imply:
\begin{align}
  \label{Eqn:TO_porous_Cylinder_Vel}
  v_r(r) = \frac{C}{r}
\end{align}
where $C$ is a constant that needs to be determined. However, it is imperative to note that $C$ depends on the permeabilities. Using equation \eqref{Eqn:TO_porous_Cylinder_Darcy}, the solution for the pressure field takes the following form:
\begin{align}
  p(r) = \left\{\begin{array}{ll}
  \frac{C \mu}{k_1} \ln(r) + C_1 & \quad r_i \leq r < \xi \\ \\
  \frac{C \mu}{k_2} \ln(r) + C_2 & \quad \xi < r \leq r_o
  \end{array}\right.
\end{align}
where $C_1$ and $C_2$ are constants. The boundary conditions \eqref{Eqn:TO_porous_Cylinder_BCi}--\eqref{Eqn:TO_porous_Cylinder_BCo} will furnish these constants, resulting in the following expression for the pressure field:
\begin{align}
  p(r) = \left\{\begin{array}{ll}
  p_i + \frac{C \mu}{k_1} \ln(r/r_i) & \quad r_i \leq r < \xi \\ \\
  p_o + \frac{C \mu}{k_2} \ln(r/r_o) & \quad \xi < r \leq r_o
  \end{array}\right.
\end{align}
The jump condition for the pressure
\eqref{Eqn:TO_porous_Cylinder_jump_pressure}
implies:
\begin{align}
  C = \frac{1}{\mu} (p_o - p_i) \Upsilon_{\mathrm{2D}}^{-1}(\xi)
\end{align}
where $\Upsilon_{\mathrm{2D}}(\xi)$, introduced for
convenience, is defined as follows:
\begin{align}
  \label{Eqn:TopOpt_Upsilon_2D}
  \Upsilon_{\mathrm{2D}}(\xi) := \frac{1}{k_1}
  \ln\left(\frac{\xi}{r_i}
  \right) + \frac{1}{k_2}
  \ln \left(\frac{r_o}{\xi}\right)
\end{align}

Based on the above solution, the total rate
of dissipation takes the following form:
\begin{align}
  \Phi(\xi) = \int_{r_i}^{r_o} \frac{\mu}{k} \left(v_r\right)^2 (2 \pi r) dr
  &= 2 \pi \mu C^2 \left(\frac{1}{k_1} \int_{r_i}^{\xi^{-}} \frac{1}{r} dr + \frac{1}{k_2} \int_{\xi^{+}}^{r_o} \frac{1}{r} dr
  \right)
  = \frac{2 \pi (p_o - p_i)^2}{\mu}
  \Upsilon_{\mathrm{2D}}^{-1}(\xi)
\end{align}
Noting that $p_i$, $p_o$ and $\mu > 0$ are
given constants and are independent of the
design variable $\xi$. Thus, maximizing
$\Phi(\xi)$ is equivalent to minimizing
$\Upsilon_{\mathrm{2D}}(\xi)$. That is,
\begin{align}
  \label{Eqn:TopOpt_2D_axisym_optim_problem}
  \widehat{\xi}_{\mathrm{2D}} \leftarrow
  \mathop{\mathrm{argmax}}_{\xi} \; \Phi(\xi) \equiv
  \mathop{\mathrm{argmin}}_{\xi} \; \Upsilon_{\mathrm{2D}}(\xi)
\end{align}
where $\widehat{\xi}_{\mathrm{2D}}$ is the optimum
location of the material interface. For further
analysis, we will denote $k_L$ and $k_H$ to denote,
respectively, the low and high permeabilities, and
consider the following two cases.

\subsubsection{High-permeability material near the inner circle}
Under this case, $k_1 = k_H$ and $k_2 = k_L$.
By substituting these values into equation
\eqref{Eqn:TopOpt_Upsilon_2D}, we define the
corresponding $\Upsilon_{\mathrm{2D}}(\xi)$ as:
\begin{align}
  \Upsilon^{(1)}_{\mathrm{2D}}(\xi) := \frac{1}{k_H}
  \ln\left(\frac{\xi}{r_i}
  \right) + \frac{1}{k_L}
  \ln \left(\frac{r_o}{\xi}\right)
  = \frac{1}{k_L} \ln(r_o) - \frac{1}{k_H}
  \ln(r_i) + \left(\frac{1}{k_H} - \frac{1}{k_L}
  \right)  \ln(\xi)
\end{align}
Since $k_H \geq k_L$, we note that $\Upsilon_{\mathrm{2D}}(\xi)$
decreases with an increase in $\xi$. The volumetric constraint
takes the following form:
\begin{align}
  \frac{2 \pi (\xi^2 - r_i^2)}{2 \pi (r_o^2 - r_i^2)} \leq \gamma
\end{align}
which implies that
\begin{align}
  \xi^2 \leq (1 - \gamma) r_i^2 + \gamma r_o^2
\end{align}
The above inequality implies that the minimum
of $\Upsilon^{(1)}_{\mathrm{2D}}(\xi)$ occurs at
\begin{align}
  \label{Eqn:TO_porous_definition_of_xi1}
  & \widehat{\xi}^{(1)}_{\mathrm{2D}}
  := \sqrt{(1 - \gamma) r_{i}^2 + \gamma r_o^2}
\end{align}
We will denote this minimum as follows:
\begin{align}
  \Upsilon_{\mathrm{2D,min}}^{(1)} :=
  \Upsilon^{(1)}_{\mathrm{2D}}\left(\widehat{\xi}^{(1)}_{\mathrm{2D}}\right)
  = \frac{1}{k_L} \ln(r_o) - \frac{1}{k_H} \ln(r_i)
  + \left(\frac{1}{k_H} - \frac{1}{k_L}\right)
  \ln\left(\widehat{\xi}^{(1)}_{\mathrm{2D}}\right)
\end{align}
The corresponding maximum rate of dissipation
will be:
\begin{align}
  \Phi^{(1)}_{\mathrm{max}} = \frac{2 \pi (p_o - p_i)^2}{\mu}
  \left(\Upsilon^{(1)}_{\mathrm{2D,min}}\right)^{-1}
\end{align}

\subsubsection{High-permeability material near the outer circle}
Under this case, $k_1 = k_L$ and $k_2 = k_H$. For
this case, we define
\begin{align}
  \Upsilon^{(2)}_{\mathrm{2D}}(\xi) := \frac{1}{k_L}
  \ln\left(\frac{\xi}{r_i}
  \right) + \frac{1}{k_H}
  \ln \left(\frac{r_o}{\xi}\right)
  = \frac{1}{k_H} \ln(r_o) - \frac{1}{k_L}
  \ln(r_i) + \left(\frac{1}{k_L} - \frac{1}{k_H}
  \right)  \ln(\xi)
\end{align}
and note that $\Upsilon^{(2)}_{\mathrm{2D}}(\xi)$ decreases
with a decrease in $\xi$. The volumetric
constraint for the second case takes the
following form:
\begin{align}
  \gamma r_i^2 + (1 - \gamma) r_o^2 \leq \xi^2
\end{align}
The above inequality implies that the maximum
of $\Upsilon^{(2)}_{\mathrm{2D}}(\xi)$ occurs at
\begin{align}
    \label{Eqn:TO_porous_definition_of_xi2}
    &\widehat{\xi}^{(2)}_{\mathrm{2D}} :=
    \sqrt{\gamma r_{i}^2 + (1 - \gamma) r_o^2}
\end{align}
We will denote this minimum by
\begin{align}
  \Upsilon_{\mathrm{2D,min}}^{(2)} :=
   \Upsilon^{(2)}_{\mathrm{2D}}\left(\widehat{\xi}_{\mathrm{2D}}^{(2)}\right)  =
  \frac{1}{k_H} \ln(r_o) - \frac{1}{k_L} \ln(r_i)
  + \left(\frac{1}{k_L} - \frac{1}{k_H}\right)
  \ln\left(\widehat{\xi}^{(2)}_{\mathrm{2D}}\right)
\end{align}
The corresponding maximum rate of dissipation
will be:
\begin{align}
  \Phi^{(2)}_{\mathrm{max}} = \frac{2 \pi (p_o - p_i)^2}{\mu}
  \left(\Upsilon^{(2)}_{\mathrm{2D,min}}\right)^{-1}
\end{align}

\begin{proposition}
  For any volumetric constraint $0 \leq \gamma \leq 1$,
  $\Upsilon_{\mathrm{2D,min}}^{(1)} \leq \Upsilon_{\mathrm{2D,min}}^{(2)}$.
\end{proposition}
\begin{proof}
  We start by considering the following difference:
  \begin{align}
    \Upsilon_{\mathrm{2D,min}}^{(2)} - \Upsilon_{\mathrm{2D,min}}^{(1)}
    &= \left(\frac{1}{k_L} - \frac{1}{k_H} \right)
    \ln\left(\widehat{\xi}^{(1)}_{\mathrm{2D}} \,
    \widehat{\xi}^{(2)}_{\mathrm{2D}}\right)
    + \left(\frac{1}{k_H} - \frac{1}{k_L}\right)
    \ln(r_i r_o) \notag \\
    &= \left(\frac{1}{k_L} - \frac{1}{k_H} \right)
    \ln\left(\frac{\widehat{\xi}^{(1)}_{\mathrm{2D}}
      \, \widehat{\xi}^{(2)}_{\mathrm{2D}}}{r_i r_o}\right)
  \end{align}
  Since $k_H \geq k_L$, we have $\left(\frac{1}{k_L}
  - \frac{1}{k_H} \right) \geq 0$. Ergo, the task at
  hand is to show $\widehat{\xi}^{(1)}_{\mathrm{2D}} \,
  \widehat{\xi}^{(2)}_{\mathrm{2D}} \geq r_i r_o$ so that
  $\ln\left(\widehat{\xi}^{(1)}_{\mathrm{2D}} \,
  \widehat{\xi}^{(2)}_{\mathrm{2D}} / (r_i r_o)\right) \geq 0$.
  Noting equations
  \eqref{Eqn:TO_porous_definition_of_xi1} and
  \eqref{Eqn:TO_porous_definition_of_xi2}, we
  now proceed as follows:
  \begin{align}
    \label{Eqn:TO_porous_Upsilon_intermediate_inequality}
    \frac{\widehat{\xi}^{(1)}_{\mathrm{2D}} \,
      \widehat{\xi}^{(2)}_{\mathrm{2D}}}{r_i r_o}
    &= \frac{1}{r_i r_o} \sqrt{ \left((1- \gamma) r_i^2 + \gamma r_o^2\right) \left((1- \gamma) r_o^2 + \gamma r_i^2 \right) } \notag \\
    &= \sqrt{ \gamma (1- \gamma) \left(
      \left(\frac{r_i}{r_o}\right)^2 + \left(\frac{r_o}{r_i}\right)^2 \right)
      + (1 - \gamma)^2 + \gamma^2 }
  \end{align}
The arithmetic-geometric mean inequality implies:
\begin{align}
  \left(\frac{r_i}{r_o}\right)^2 + \left(\frac{r_o}{r_i}\right)^2
  \geq 2
\end{align}
Furthermore, noting that $0 \leq \gamma \leq 1$ (which implies $\gamma (1 - \gamma) \geq 0$), equation \eqref{Eqn:TO_porous_Upsilon_intermediate_inequality} will give rise to the following inequality:
\begin{align}
  \frac{\widehat{\xi}^{(1)}_{\mathrm{2D}} \,
    \widehat{\xi}^{(2)}_{\mathrm{2D}}}{r_i r_o}
  \geq \sqrt{ 2 \gamma (1- \gamma) + (1 - \gamma)^2 + \gamma^2 } \equiv 1
\end{align}
which completes the proof.
\end{proof}

On the account of equation \eqref{Eqn:TopOpt_2D_axisym_optim_problem},
the above proposition implies that $\Phi_{\max}^{(1)} \geq \Phi_{\max}^{(2)}$.
Hence, we have shown that the optimal layout of the material for pressure-driven problem with cylindrical symmetry is to place the high permeability material near the inner boundary. The optimal location of the material interface is
\begin{align}
  \label{Eqn:TopOpt_2D_axisym_xiopt}
  \widehat{\xi}_{\mathrm{2D}} = \widehat{\xi}^{(1)}_{\mathrm{2D}}
  = \sqrt{(1 - \gamma) r_{i}^2 + \gamma r_o^2}
\end{align}

\subsection{Velocity-driven}
We now analyze the same two-dimensional domain
comprising two concentric circles but with the
following velocity boundary condition on the
inner boundary:
\begin{align}
  \label{Eqn:Axisym_Cylinder_Vel_BCi}
  &v_r(r = r_i) = v_o
\end{align}
The governing equations will remain as that
of the pressure-driven problem expect for
equation \eqref{Eqn:TO_porous_Cylinder_BCi}
is replaced with the above boundary condition.
The jump conditions at the material interface
(i.e., equations
\eqref{Eqn:TO_porous_Cylinder_jump_pressure} and \eqref{Eqn:TO_porous_Cylinder_jump_velocity})
remain the same.
Equation \eqref{Eqn:TO_porous_Cylinder_Vel}
along with the boundary condition
\eqref{Eqn:Axisym_Cylinder_Vel_BCi} and the
velocity jump condition
\eqref{Eqn:TO_porous_Cylinder_jump_pressure}
imply:
\begin{align}
  v_r(r) = v_o \, \frac{r_i}{r}
\end{align}
The total rate of dissipation is:
\begin{align}
  \Phi(\xi) = \int_{r_i}^{r_o} \frac{\mu}{k} \left(v_r\right)^2 (2 \pi r) dr
  &= 2 \pi \mu v_o^2 r_i^2 \left(\frac{1}{k_1} \int_{r_i}^{\xi^{-}} \frac{1}{r} dr + \frac{1}{k_2} \int_{\xi^{+}}^{r_o} \frac{1}{r} dr\right)
	= 2 \pi \mu v_o^2 r_i^2 \Upsilon_{\mathrm{2D}}(\xi)
\end{align}

For velocity-driven problems, the design problem is
posed as the minimization of the rate of dissipation
with a volume constraint bound on the high-permeability
material. Since, $\Phi(\xi) \propto \Upsilon_{\mathrm{2D}}(\xi)$,
the minimization of $\Phi(\xi)$ is equivalent to the
minimization of $\Upsilon_{\mathrm{2D}}(\xi)$. Recall that we have
minimized $\Upsilon_{\mathrm{2D}}(\xi)$ even for the pressure-driven
problem. (However, in the case of the pressure-driven
problem, $\Phi(\xi) \propto \Upsilon^{-1}_{\mathrm{2D}}(\xi)$ and we
maximized $\Phi(\xi)$.) Thus, the optimal material
distribution for the velocity-driven problem will be
the same as that of the pressure-driven problem---the
high-permeability material should be placed near the inner
boundary. The location of the material interface will be
the same as for the pressure-driven problem (see equation
\eqref{Eqn:TopOpt_2D_axisym_xiopt}).

\subsection{Numerical results}
\begin{table}
	\caption{Parameters used in the numerical simulation of
	the 2D axisymmetric problem---concentric cylinders.
	\label{Fig:TO_porous_cylinder_BVP_parameters}}
	\begin{tabular}{|lr||lr|}\hline
		parameter & value & parameter & value \\ \hline
		$r_i$ & 0.1 & $r_o$ & 1.0 \\
		$p_i$ & 100 & $p_o$ & 1 \\
		$k_L$ & 1 & $k_H$ & 10 \\
		$\gamma$ & 0.3 & $\mu$ & 1 \\ \hline
	\end{tabular}
\end{table}

\begin{figure}[h]
	\subfigure[$\gamma=0.1$]{\includegraphics[scale=0.4]{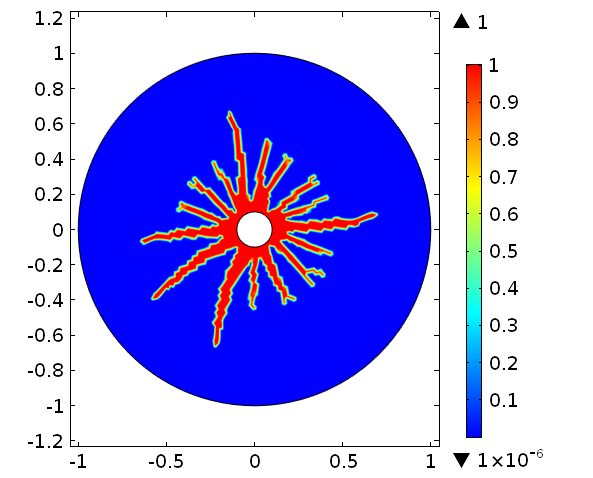}}
	\hspace{0.5 cm}
	\subfigure[$\gamma=0.5$]{\includegraphics[scale=0.4]{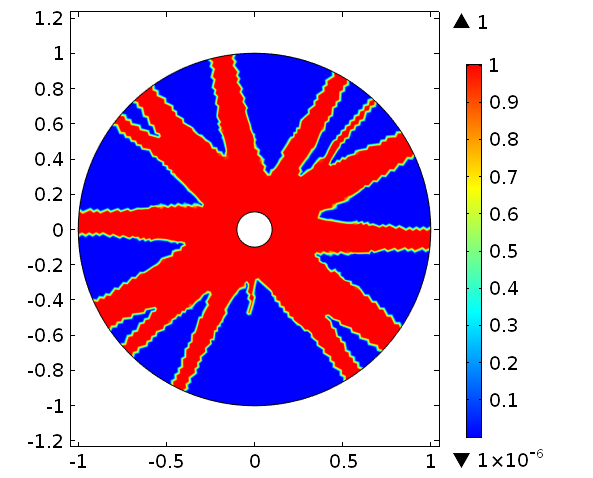}}
	\caption{This figure shows the material distribution under the \emph{Darcy model}, obtained by solving both the ingredients of the design problem---the primal problem and the optimization procedure---numerically. The numerical scheme for the primal analysis did not enforce radial symmetry explicitly. That is, the primal analysis is a two-dimensional finite element analysis using triangular elements rather than axisymmetric finite elements. \label{Conc_cyln1}}
\end{figure}

\begin{figure}[h]
	\subfigure[$\gamma=0.1$]{\includegraphics[scale=0.4]{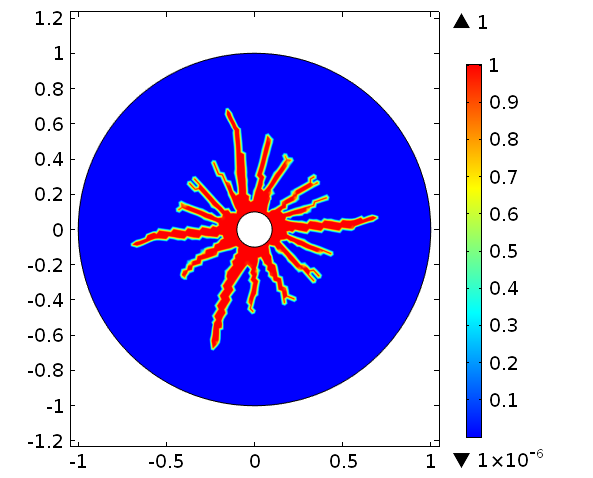}}
	\hspace{0.5cm}
	\subfigure[$\gamma=0.5$]{\includegraphics[scale=0.4]{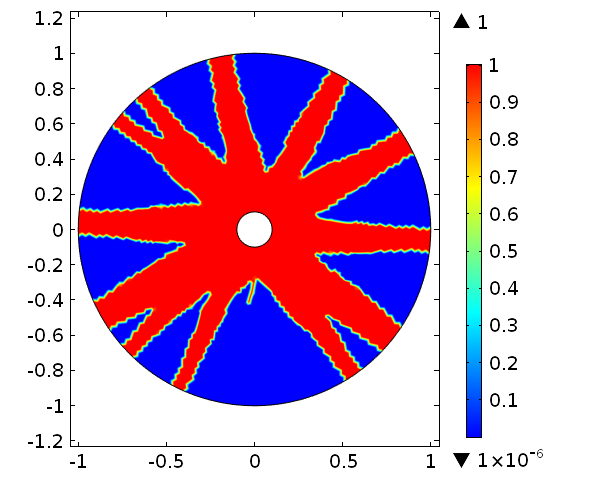}}
        \caption{This figure shows the material distribution under the \emph{Barus model}, obtained by solving both the ingredients of the design problem---the primal problem and the optimization procedure---numerically. The numerical scheme for the primal analysis did not enforce radial symmetry explicitly. That is, the primal analysis is a two-dimensional finite element analysis using triangular elements rather than axisymmetric finite elements. \label{Conc_cyln2}}
\end{figure}

Numerical solutions are obtained for the above design problem; both the primal analysis and the optimization procedure are solved numerically. Since the selection and accuracy of numerical schemes for primal analysis can affect the final design outcome, we compare the designs under two different approaches for the primal analysis. The first approach is a two-dimensional finite element analysis using triangular elements; the underlying radial symmetry is not built into the analysis. The second approach uses axisymmetric finite elements, and hence the radial symmetry is invoked in the primal analysis.

As shown in figures \ref{Conc_cyln1} and \ref{Conc_cyln2}, \textsf{TopOpt} without invoking symmetry yields a material distribution in the form of dendrites or fingers. Despite using regularization (e.g., SIMP) and high-order discretizations, these finger-type design patterns are observed. However, using axisymmetric analysis, we could obtain a discrete 0--1 material distribution with high-permeability material being placed perfectly concentric and near to the inner circle; see figure \ref{Fig:TO_porous_cylinder_axi_material_layout}.

\begin{figure}[h]
  \subfigure[Darcy model]{\includegraphics[scale=0.3]{Figures/Axisymmetric_problems/Concentric_cylinders/1D_Axisymmetric_Perfect.png}}
  \hspace{0.5cm}
  \subfigure[Barus model]{\includegraphics[scale=0.3]{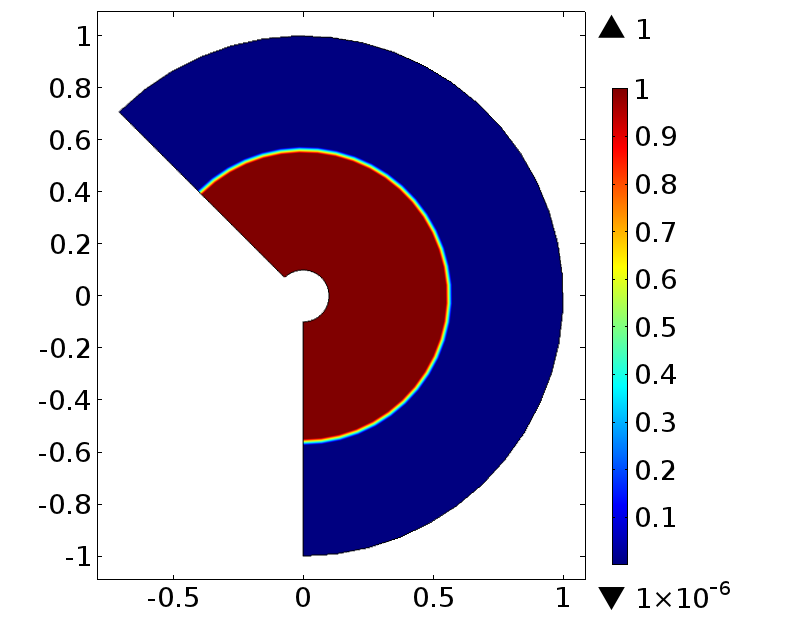}}
  \caption{This figure shows the optimal material distributions for $\gamma$ = 0.3 under the Darcy and Barus models. The primal analysis invoked axisymmetry; that is, axisymmetric finite elements are employed. \label{Fig:TO_porous_cylinder_axi_material_layout}}
\end{figure}

\begin{figure}[h]
  \subfigure[Darcy model]{\includegraphics[scale=0.3]{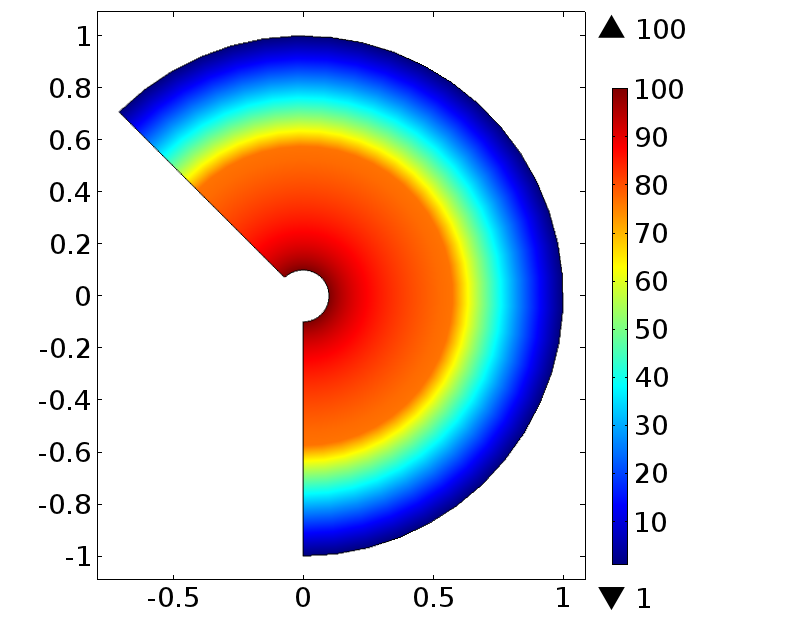}}
  \hspace{0.5cm}
  \subfigure[Barus model]{\includegraphics[scale=0.3]{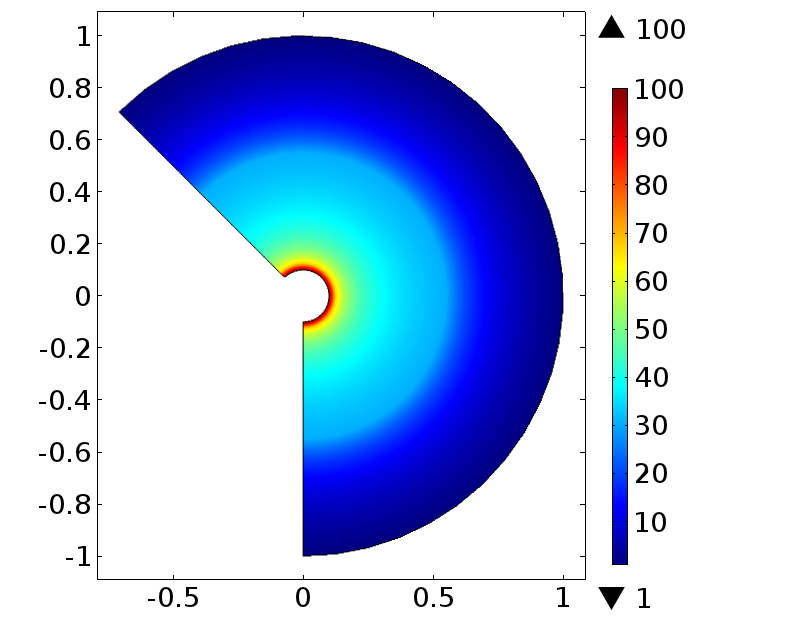}}
  \caption{Pressure distribution for $\gamma$ = 0.3 under different models
  corresponding to the material layout given in figure \ref{Fig:TO_porous_cylinder_axi_material_layout}. \label{Conc_cyln4}}
\end{figure}

\begin{figure}[h]
	\subfigure[Darcy model]{\includegraphics[scale=0.3]{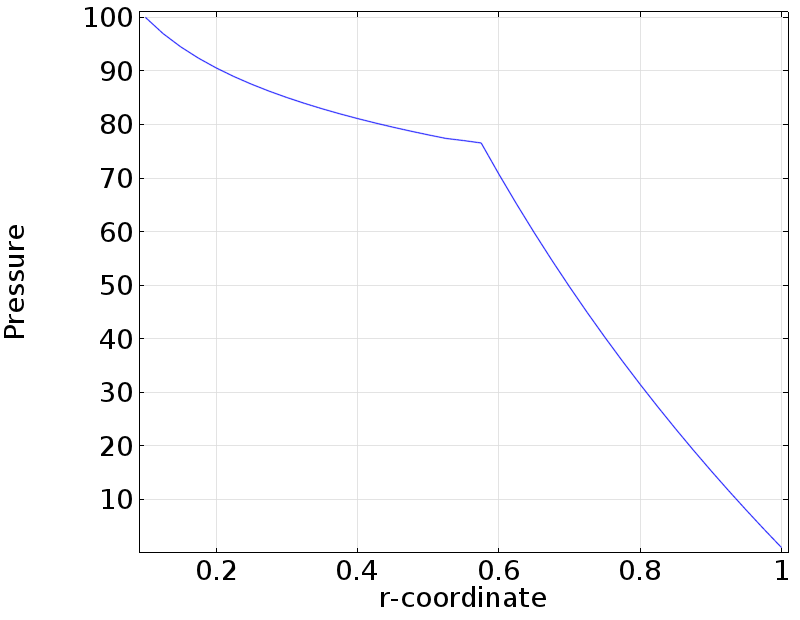}}
	\hspace{0.5cm}
	\subfigure[Barus model]{\includegraphics[scale=0.3]{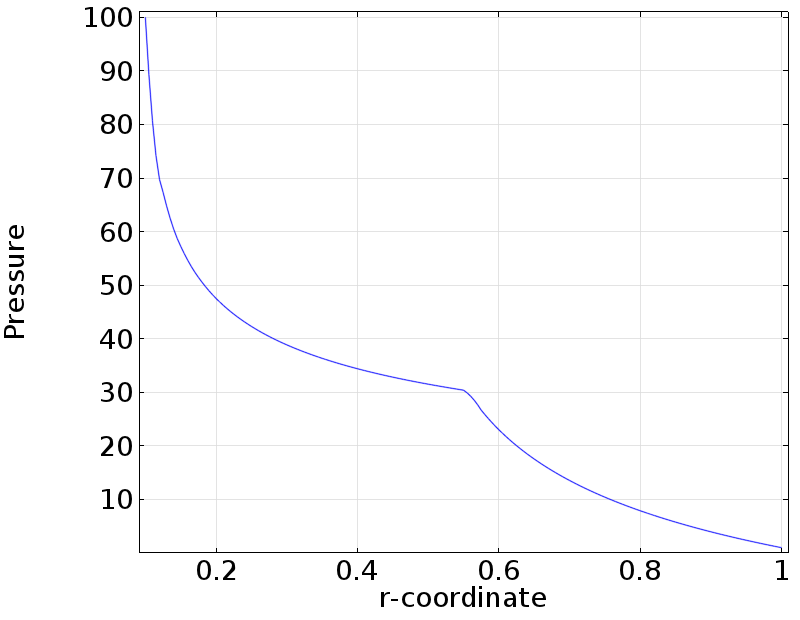}}
	\caption{Graph of pressure versus the radius for $\gamma$ = 0.3 under different models
		corresponding to the material layout given in figure \ref{Fig:TO_porous_cylinder_axi_material_layout}. \label{Conc_cyln5}}
\end{figure}

Several reasons suggest that the dendrite-type solution is a numerical artifact, and hence, \emph{unphysical}:
\begin{enumerate}[(a)]
\item Since the boundary conditions and the geometry are all symmetric, the optimal material layout from the numerical simulation should respect the underlying symmetry. But the dendrite-type solution does not respect the radial symmetry.
\item This numerical solution did not match with the analytical solution, which is derived with reasonable assumptions. 
\item The optimization procedure took over 350 iterations for the dendrite-type solution, while it took less than 25 iterations for the symmetric solution reported in figure \ref{Fig:TO_porous_cylinder_axi_material_layout}.
\item Lastly, such designs with sharp irregular features, even if provided by \textsf{TopOpt}, are often treated as impractical designs and are rejected based on manufacturability; these designs are challenging to fabricate---even with additive manufacturing techniques.
\end{enumerate}

The possibility of finger-type solutions or checkered board patterns under topology optimization is adequately discussed in the literature; for example, see \citep{bruns2007topology,bendsoe_sigmund_2013topology}. Herein, we have shown that such pathological numerical designs for an axisymmetric problem could arise if one does not explicitly enforce the underlying symmetry---by not using axisymmetric finite elements. The other main findings of this section are:
\begin{enumerate}[(i)]
\item The numerical solution using axisymmetric finite elements matches the analytical solution, thereby justifying the single material interface assumption used in the analytical solution.
\item Under an optimal solution for a 2D axisymmetric problem, the high-permeability material should be placed near the inner boundary. This material distribution in this problem is different from that of the 1D problem, where high-permeability material can be placed anywhere.
  \item The non-linearity under the Darcy-Barus model does not seem to affect the optimal material distribution for the chosen problem; however, the pressure distribution within the domain under the Darcy and Darcy-Barus models vastly differs, as shown in figures \ref{Conc_cyln4} and \ref{Conc_cyln5}. Although the inlet and outlet pressures remain the same, the internal pressure within the domain evaluated using the Darcy-Barus model is lower than the corresponding pressure under the Darcy model.
\end{enumerate}

\section{OPTIMAL LAYOUTS FOR A 3D AXISYMMETRIC PROBLEM}
\label{Sec:S7_Porous_AxiSym_3D}

Consider a three-dimensional domain comprising two concentric spheres with an inner radius $r_i$ and an outer radius $r_o > r_i$. A pressure $p_i$ is applied on the inner boundary, and the outer boundary is subject to a pressure $p_o < p_i$. The design question is: \emph{Given two materials with permeabilities $k_1$ and $k_2$, what is the optimal material distribution of these two materials in the domain.} The optimality is judged based on the maximization of the rate of dissipation with a volume constraint bound on the high-permeability material. We will first present the governing equations for the primal analysis. Next, we will obtain an analytical solution for the design problem. Finally, we will obtain numerical optimal solutions using \textsf{TopOpt} and compare them with the analytical solution. 

\subsection{Primal analysis}
For mathematical convenience, we will
use the spherical polar coordinates: 
\begin{align}
  r_i \leq r \leq r_o, \quad
  0 \leq \theta \leq \pi, \quad 
  0 \leq \phi \leq 2\pi 
\end{align}
where $r$ is the radial distance from the center,
and $\theta$ and $\phi$ are the azimuth and polar
angles, respectively. We assume that the body force
is zero: $\mathbf{b}(\mathbf{x}) = \mathbf{0}$. The
permeability is assumed to vary only in the radial
direction. That is, 
\begin{align}
  k(\mathbf{x}) = k(r) 
\end{align}
Noting the inherent symmetry in the
problem, the velocity and pressure
fields can be written as follows:
\begin{align}
  \mathbf{v}(\mathbf{x}) = v_{r}(r) \, \hat{\mathbf{e}}_{r}
  \quad \mathrm{and} \quad
  p(\mathbf{x}) = p(r) 
\end{align}
where $v_{r}$ is radial component of the velocity,
and $\hat{\mathbf{e}}_r$ is the unit vector along
the radial direction. The Darcy model is used to
describe the flow. 

\subsubsection{Non-dimensionalization}
We take the following (pressure, length
and viscosity) as reference quantities:
\begin{align}
  p_\mathrm{ref} = p_i - p_o \;
  [\mathrm{M}^{1} \mathrm{L}^{-1} \mathrm{T}^{-2}], \quad 
  L_\mathrm{ref} = r_o \; 
  [\mathrm{M}^{0}\mathrm{L}^{1}\mathrm{T}^{0}] 
  \quad \mathrm{and} \quad 
  \mu_\mathrm{ref} = \mu \; 
     [\mathrm{M}^{1} \mathrm{L}^{-1} \mathrm{T}^{-1}]
\end{align}
These references quantities give rise to the
following non-dimensional quantities:
\begin{align}
  \bar{r} = \frac{r}{r_o}, \bar{r}_i = \frac{r_i}{r_o}, \bar{r}_o = 1, \; \bar{k} = \frac{k}{r^2_o}, \; 
  \bar{p} = \frac{p - p_o}{p_i - p_o}, \; \bar{p}_i = 1, \; \bar{p}_o = 0, \; 
  \bar{\mu} = 1, \; \bar{v}_{r} = \frac{v_{r} \mu}{r_o (p_i - p_o)} 
\end{align}
The governing equations, in non-dimensional form, 
for the primal analysis can be written as follows:
\begin{subequations}
  \begin{alignat}{2}
    \label{Eqn:TO_porous_Sphere_Darcy}
    &\frac{1}{\bar{k}(\bar{r})} \bar{v}_r
    + \frac{d \bar{p}}{d \bar{r}} = 0
    &&\quad \forall \bar{r} \in (\bar{r}_i, 1) \\
    \label{Eqn:TO_porous_Sphere_Continuity}
    &\frac{1}{\bar{r}^2} \frac{d (\bar{r}^2 \bar{v}_r)}{d \bar{r}} = 0
    &&\quad \forall \bar{r} \in (\bar{r}_i, 1) \\
    \label{Eqn:TO_porous_Sphere_BCi}
    &\bar{p}(\bar{r} = \bar{r}_i) = 1 \\
    \label{Eqn:TO_porous_Sphere_BCo}
    &\bar{p}(\bar{r} = 1) = 0
	\end{alignat}
\end{subequations}
From hereon we will drop the superposed overlines for
convenience; all the quantities in the rest of this
section should be considered as non-dimensional. 

\subsection{Analytical solution for the design problem}
We make the same assumption as we made for
the 2D axisymmetric problem (\S\ref{Sec:S6_Porous_AxiSym_2D}): 
\emph{each material is present in a contiguous and symmetric manner, and a single boundary exists between the two materials over the entire domain.} This assumption allows us to write the spatial dependence of the permeability as follows: 
\begin{align}
  k(r) = \left\{\begin{array}{ll}
  k_1 & r_i < r < \xi \\
  k_2 & \xi < r < 1
  \end{array} \right.
\end{align}
where $r = \xi$ is the location of the material
interface. The jump conditions along the interface
take the following form:
\begin{subequations}
  \begin{align}
    \label{Eqn:TopOpt_3D_velocity_jump}
    &v_r(r = \xi^{-}) = v_r(r = \xi^{+}) \\
    \label{Eqn:TopOpt_3D_pressure_jump}
    &p(r = \xi^{-}) = p(r = \xi^{+})
  \end{align}
\end{subequations}

Using equations \eqref{Eqn:TO_porous_Sphere_Continuity}
and \eqref{Eqn:TopOpt_3D_velocity_jump} and noting that
$r$ $\neq 0$, the velocity field can be written as follows: 
\begin{align}
  \label{Eqn:TO_porous_Sphere_Velocity}
  v(r) = \frac{A}{r^2} 
\end{align}
where $A$ is a constant to be determined. Using
equations \eqref{Eqn:TO_porous_Sphere_Velocity}
and \eqref{Eqn:TO_porous_Sphere_Darcy}, the
solution for the pressure field is:
\begin{align}
  p(r) = \left\{ \begin{array}{cc} 
    \frac{1}{k_1 r} A + B_1 & \\ \\
    \frac{1}{k_2 r} A + B_2 &
    \end{array} \right.
\end{align}
where $B_1$ and $B_2$ are constants. The pressure
boundary conditions \eqref{Eqn:TO_porous_Sphere_BCi}
and \eqref{Eqn:TO_porous_Sphere_BCo} and the jump
condition \eqref{Eqn:TopOpt_3D_pressure_jump} will
establish:
\begin{align}
  A = \left[\frac{1}{k_1 r_i} - \frac{1}{k_2}
    + \frac{1}{\xi} \left(\frac{1}{k_2} - \frac{1}{k_1}
    \right)\right]^{-1}, \quad 
  B_1 = 1 - \frac{A}{k_1 r_i}
  \quad \mathrm{and} \quad 
  B_2 = -\frac{A}{k_2} 
\end{align}
Thus, the final expression for the velocity field is:
\begin{align}
  v_r(r) = \Upsilon_{\mathrm{3D}}(\xi) \frac{1}{r^2} =
  \left[\frac{1}{k_1 r_i} - \frac{1}{k_2}
    + \frac{1}{\xi} \left(\frac{1}{k_2} - \frac{1}{k_1}
    \right)\right]^{-1} \frac{1}{r^2}
\end{align}
Consequently, the total rate of dissipation
takes the following form:
{\small
\begin{align}
  \Phi(\xi) = \int_{r_i}^{1} \frac{4\pi r^2}{k(r)} \; v_r^2 \; dr 
  & = 4 \pi A^2 \left(\int_{r_i}^{\xi} \frac{1}{k_1} \, \frac{1}{r^2} \; dr
  + \int_{\xi}^{1} \frac{1}{k_2} \, \frac{1}{r^2} \; dr \right)
  = 4 \pi \Upsilon_{\mathrm{3D}}(\xi)
\end{align}
}

Under the single material interface, the design
problem -- maximizing the rate of dissipation with
a volume constraint bound on the high-permeability
material -- reduces to answering the following
question:
\begin{quote}
  \emph{Which of the two cases---placing the
    high-permeability material near the inner
    boundary or near the outer boundary---gives
    the higher rate of dissipation?}
\end{quote}
To answer this question, we will now explore these
two cases separately. For further analysis, we will
denote the high and low permeabilities by $k_H$ and
$k_L$, respectively, with  $k_L \leq k_H$. We denote
the optimal location of the material interface by
$\widehat{\xi}$. 

\emph{Case 1: High-permeability material
  near the inner boundary.} Under this case, $k_1 = k_H$
and $k_2 = k_L$. The corresponding rate of dissipation is:
\begin{align}
  \Phi^{(1)} = 4 \pi \left[ \frac{1}{k_H r_i}
    - \frac{1}{k_L} + \frac{1}{\xi}
    \left(\frac{1}{k_L} - \frac{1}{k_H} \right) 
    \right]^{-1}
\end{align}
The volumetric constraint takes the following form: 
\begin{align}
  \frac{4 \pi}{3} (\xi^3 - r_i^3) \;  \leq \;
  \gamma \frac{4 \pi}{3} (1 - r_i^3)
\end{align}
which implies that 
\begin{align}
  \xi^3 \leq \gamma + (1 - \gamma) r_i^3
\end{align}
(Recall that $\gamma$ is the bound in the volume
constraint for the high-permeability material.)
Since $k_{L} \leq k_{H}$, $\Phi^{(1)}$ increases
monotonically with $\xi$. Thus, for this case,
the maximum rate of dissipation occurs for 
\begin{align}
  \label{Eqn:TO_porous_Sphere_gamma_in}
  \widehat{\xi} = \widehat{\xi}_{\mathrm{3D}}^{(1)}
  := \sqrt[3] { \gamma + (1 - \gamma) r_i^3}
\end{align}
with the corresponding value: 
\begin{align}
  \label{Eqn:3Daxi_Phi_1_max}
  \Phi^{(1)}_{\mathrm{max}} = 4 \pi \left[ \frac{1}{k_H r_i}
    - \frac{1}{k_L} + \frac{1}{\widehat{\xi}^{(1)}_{\mathrm{3D}}}
    \left(\frac{1}{k_L} - \frac{1}{k_H} \right) 
    \right]^{-1}
\end{align}

\emph{Case 2: High-permeability material
  near the outer boundary.}
Under this case, $k_1 = k_L$ and $k_2 = k_H$. The
corresponding rate of dissipation is:
\begin{align}
  \Phi^{(2)} = 4 \pi \left[ \frac{1}{k_L r_i}
    - \frac{1}{k_H} + \frac{1}{\xi}
    \left(\frac{1}{k_H} - \frac{1}{k_L} \right) 
    \right]^{-1}
\end{align}
The volumetric constraint takes the following form: 
\begin{align}
  \frac{4 \pi}{3} (r_o^3 - \xi^3)
  \;  \leq \;  \gamma \frac{4 \pi}{3} (1 - r_i^3)
\end{align}
which implies that 
\begin{align}
  (1 - \gamma) + \gamma r_i^3
  \leq \xi^3 
\end{align}
Using that $k_L \leq k_H$, we conclude that
$\Phi^{(2)}$ decreases monotonically with $\xi$.
Thus, the maximum rate of dissipation occurs for 
\begin{align}
  \label{Eqn:TO_porous_Sphere_gamma_out}
  \widehat{\xi} = \widehat{\xi}_{\mathrm{3D}}^{(2)}
  := \sqrt[3] {(1 - \gamma) + \gamma r_i^3}
\end{align}
with the corresponding value: 
\begin{align}
  \label{Eqn:3Daxi_Phi_2_max}
  \Phi^{(2)}_{\mathrm{max}} = 4 \pi \left[ \frac{1}{k_L r_i}
    - \frac{1}{k_H} + \frac{1}{\widehat{\xi}^{(2)}_{\mathrm{3D}}}
    \left(\frac{1}{k_H} - \frac{1}{k_L} \right) 
    \right]^{-1}
\end{align}

To answer the question of which of these two
cases furnishes the optimal solution, we will
use equations \eqref{Eqn:3Daxi_Phi_1_max}
and \eqref{Eqn:3Daxi_Phi_2_max} and consider
the following difference:
\begin{align}
  \label{Eqn:TO_porous_Phi_intermediate_inequality}
  \frac{1}{\Phi^{(1)}_{\mathrm{max}}} -
  \frac{1}{\Phi^{(2)}_{\mathrm{max}}}  
  &= - \frac{1}{4 \pi} \left(\frac{1}{k_L} - \frac{1}{k_H}\right)
  \left(1 + \frac{1}{r_i} -
  \left( \frac{1}{\widehat{\xi}_{\mathrm{3D}}^{(1)}}
  + \frac{1}{\widehat{\xi}_{\mathrm{3D}}^{(2)}} \right) \right)
\end{align}
Since $k_H \geq k_L$, Lemma
\ref{Lemma:Lemma_on_xi_stars}
implies that
\begin{align}
  \frac{1}{\Phi^{(1)}_{\mathrm{max}}} -
  \frac{1}{\Phi^{(2)}_{\mathrm{max}}}
  \leq 0 
\end{align}
This father implies that 
\begin{align}
  \Phi^{(1)}_{\mathrm{max}} \geq \Phi^{(2)}_{\mathrm{max}}
  \quad \forall \gamma \in [0,1]
\end{align}
Thus, the optimal solution is the placement of
high permeability material near the inner boundary
satisfying the volume constraint while the rest of
the domain is occupied by the low permeability
material. For a given volume constraint bound
$\gamma$, the optimal location of the material
interface is 
\begin{align}
  \widehat{\xi}_{\mathrm{3D}} = \widehat{\xi}_{\mathrm{3D}}^{(1)}
  = \sqrt[3]{\gamma + (1 - \gamma)r_i^3}
\end{align}

\subsection{Numerical designs using TopOpt}
To validate the \emph{single material interface} assumption, we obtained a numerical solution for the design problem. Again, the primal analysis as well as the optimization procedure are solved numerically. We enforced the underlying radial symmetry into the primal analysis by using two-dimensional triangular axisymmetric finite elements.
Figure \ref{2D_Axisymmetric} shows the numerical
optimal material layouts from \textsf{TopOpt}
using the input parameters given in Table
\ref{Fig:TO_porous_sphere_BVP_parameters}.
Similar to the 2D axisymmetric analysis
(\S\ref{Sec:S6_Porous_AxiSym_2D}) and in agreement
with the analytical solution, the high-permeability
material (represented by `red') is placed near the
inner boundary while the low-permeability material
(represented by `blue') is placed near the outer
boundary. This numerical solution also justifies
the use of the single interface assumption in
obtaining the analytical solution.

\begin{table}[h]
  \caption{Parameters used in the numerical simulation
    of the 3D axisymmetric problem---concentric spheres.
    \label{Fig:TO_porous_sphere_BVP_parameters}}
  \begin{tabular}{|lr||lr|}\hline
		parameter & value & parameter & value \\ \hline
		$r_i$ & 0.1 & $r_o$ & 1.0 \\
		$p_i$ & 100 & $p_o$ & 1 \\
		$k_L$ & 1 & $k_H$ & 10 \\
		$\gamma$ & 0.1 & $\mu$ & 1 \\ \hline
	\end{tabular}
\end{table}

\begin{figure}[h]
  \subfigure[Material distribution]{\includegraphics[scale=0.4]{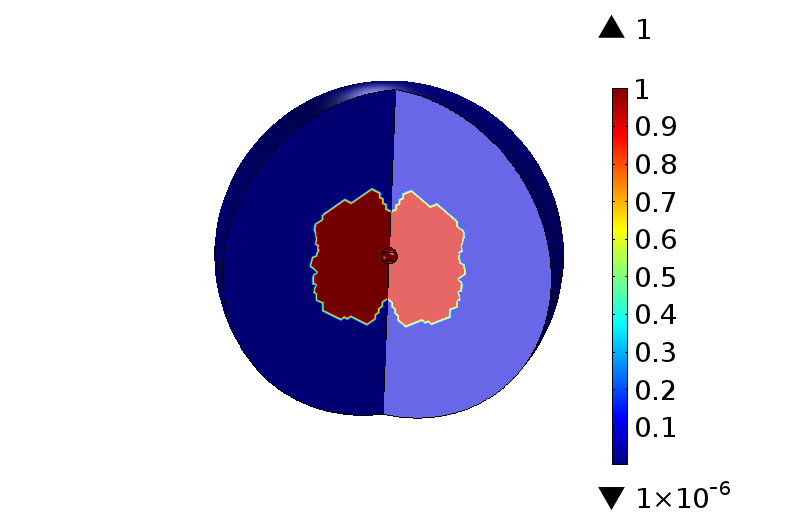}}
  %
  \subfigure[Pressure distribution]{\includegraphics[scale=0.4]{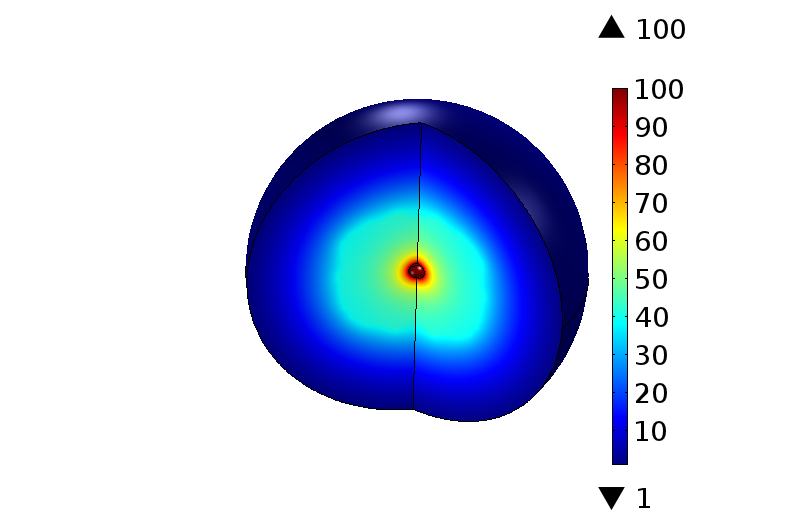}}
  \caption{The top figure shows the material
    layout obtained using \textsf{TopOpt} for
    the 3D axisymmetric problem under the Darcy
    model. The bottom figure shows the corresponding
    pressure distribution in the domain.
    \label{2D_Axisymmetric}}
\end{figure}

\section{EFFECT OF PRESSURE-DEPENDENT VISCOSITY ON OPTIMAL LAYOUTS}
\label{Sec:S8_Porous_Barus}
The aim of this section is to compare optimal solutions under the Darcy and linearized Barus models. Specifically, we study how $\beta_B$ and aspect ratio of the domain affect the optimal material layouts and the resulting velocity and pressure profiles. To see significant differences, we will allow for non-zero volumetric source. Then the continuity equation \eqref{Eqn:TopOpt_GE_Continuity} should be altered as follows:
\begin{align}
  \mathrm{div}[\mathbf{v}] = Q \quad \mathrm{in} \; \Omega
\end{align}
where $Q \; [\mathrm{M}^{0}\mathrm{L}^{0}\mathrm{T}^{-1}]$ is the strength of the volumetric source; a positive value for $Q$ means source while a negative value means sink.

We will consider two porous materials with different permeabilities---material 1 with high permeability ($k_H$) and material 2 with low permeability ($k_L$). We will find the optimal layout of these two materials within the domain by maximizing the total mechanical dissipation rate with a bound on the volume fraction of the high-permeability material. Table \ref{Fig:Pressure-driven_BVP_parameters} provides the values of the parameters used in the numerical simulations.

\begin{table}
  \caption{Parameters used in the numerical simulation
    for the pressure-driven and pipe-bend problems.
    \label{Fig:Pressure-driven_BVP_parameters}}
	\begin{tabular}{|lr||lr|}\hline
		\textbf{parameter} & \textbf{value} & \textbf{parameter} & \textbf{value} \\ \hline
		Inlet pressure, $p_i$ & 100 & Outlet pressure, $p_o$ & 1 \\
		High permeability, $k_H$ & 1 & Low permeability, $k_L$ & 0.1 \\
		Volumetric constraint, $\gamma$ & 0.1 & Fluid viscosity, $\mu$ & 1 \\
		Volumetric source, Q & 0, 1 &
		Body force, $\rho \mathbf{b}$ & $\mathbf{0}$ \\ \hline
	\end{tabular}
\end{table}

  The numerical simulation's main inputs are the geometry, the boundary conditions, properties of the two constituent materials, and fluid properties (e.g., Barus coefficient). By placing an inequality constraint on the usage of the control material using the volume bound constraint, $\gamma$, we seek the optimal layout of the material distribution for the specified inputs. Topology optimization for a specific objective function (maximization of the rate of dissipation in this case) is carried out using the homogenization method with control variable, $\xi$, where $0 \leq \xi \leq 1$.

\subsection{Pressure-driven problem} Consider a 2D computational domain of dimensions $2 \times 1.5$. (The problem is presented in a non-dimensional form.) A segment of the left boundary, $\Gamma^p$, is subject to a pressure boundary condition, $p_i=100$, while a segment of the right face is subject to a lower pressure boundary condition, $p_o=1$. The rest of the boundary, $\Gamma^v$, is subject to homogeneous velocity boundary condition (i.e., $\mathbf{v}(\mathbf{x}) \cdot \widehat{\mathbf{n}}(\mathbf{x}) = 0$ on $\Gamma^{v}$). The computational domain along with the boundary conditions are shown in figure \ref{Fig:Rect_domain2}.

Figures \ref{Fig:Pressure_driven_zero_Q} and \ref{Fig:Pressure_driven_nonzero_Q} show the results (i.e., material distribution, and the corresponding solution fields) for zero volumetric source ($Q = 0$) and non-zero volumetric source ($Q = 1$), respectively, for various values of the Barus coefficient $\beta_B$, including $\beta_B = 0$, which corresponds to the Darcy model. The main observations for the \emph{zero} volumetric source case are:
\begin{enumerate} 
\item The optimal material distribution is unaffected by the value of the Barus coefficient. Specifically, the high-permeability material is placed along the flow path of the fluid from inlet to the outlet, irrespective of the value of the Barus coefficient.
\item However, the corresponding magnitude of the velocity and the spatial distribution of the pressure differ significantly.
\end{enumerate}
For the case of \emph{non-zero} volumetric source case,
the main observations are:
\begin{enumerate}
\item The optimal material distribution varies with the Barus coefficient.
\item The Barus model yields a much different material layout than the Darcy model. In contrast to the Darcy model, the material distribution under the Barus model does not necessarily follow the flow path.
\item Notably, the qualitative nature of the material distribution under the Barus model could change from problem to problem and for different values of the Barus coefficients; this change is because of the nonlinearity in the Barus model. Figures \ref{Fig:Nonzero_Q_DarcyExplanation} and \ref{Fig:Nonzero_Q_BarusExplanation} support this claim.
  \begin{enumerate}
  \item Since viscosity and permeability are independent of the solution fields under the Darcy model, the magnitude of the velocity is proportional to the magnitude of the pressure. So, the high-permeability material under the Darcy model is placed in the regions with more significant pressure gradients. See figure \ref{Fig:Nonzero_Q_DarcyExplanation}.
\item For the same problem under the Barus model, the high-permeability material is placed in the regions with high-pressure gradients for smaller values of $\beta_B$. But for a higher value of $\beta_B$, the high-permeability material is placed in the regions with smaller pressure gradients (cf. figure \ref{Fig:Nonzero_Q_BarusExplanation}). It is, therefore, not possible to predict the optimal material layout with actually solving the design problem, as the layout depends on pressure (via pressure-dependent viscosity) as well as the pressure gradient. In general, the relative strengths---of pressure and pressure gradient---will not be known \emph{a priori}. 
  \end{enumerate}
\end{enumerate}

\begin{figure}[h]
	\includegraphics[scale=0.65]{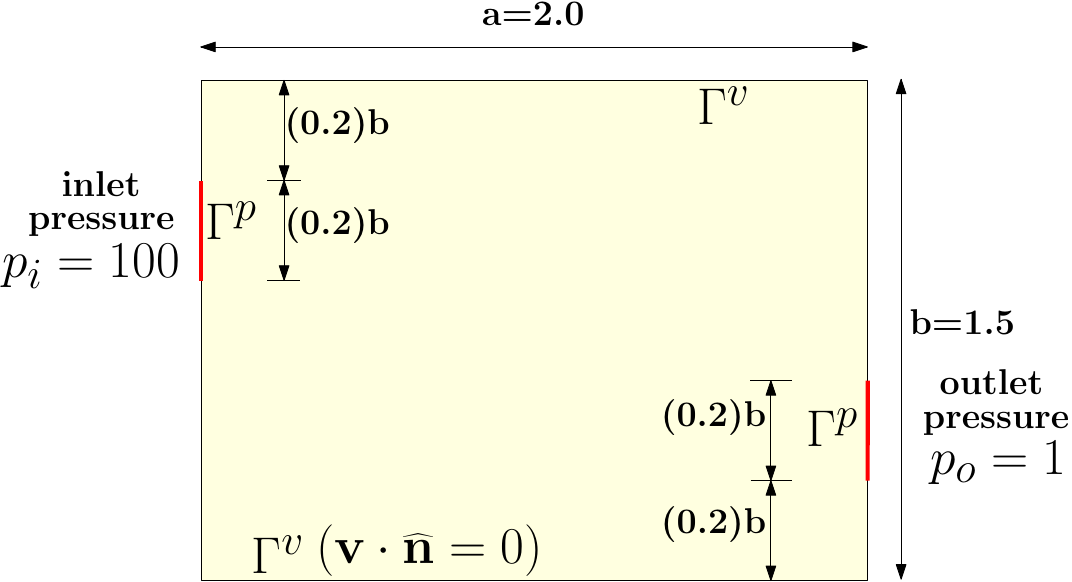}
	\caption{\textsf{Pressure-driven problem:} The flow of a fluid in a 2D rectangular porous domain (2.0 x 1.5) is driven by pressure boundary conditions. An inlet pressure $p_i = 100$ is applied on a part of the left face, and an outlet pressure $p_o = 1$ is applied on a part of the right face. On the rest of the boundary, the normal component of the velocity is prescribed to be zero. A uniform volumetric source $Q$ is applied over the entire domain. \label{Fig:Rect_domain2}}
\end{figure}

\begin{figure}[h]
  \subfigure{\includegraphics[scale=0.25]{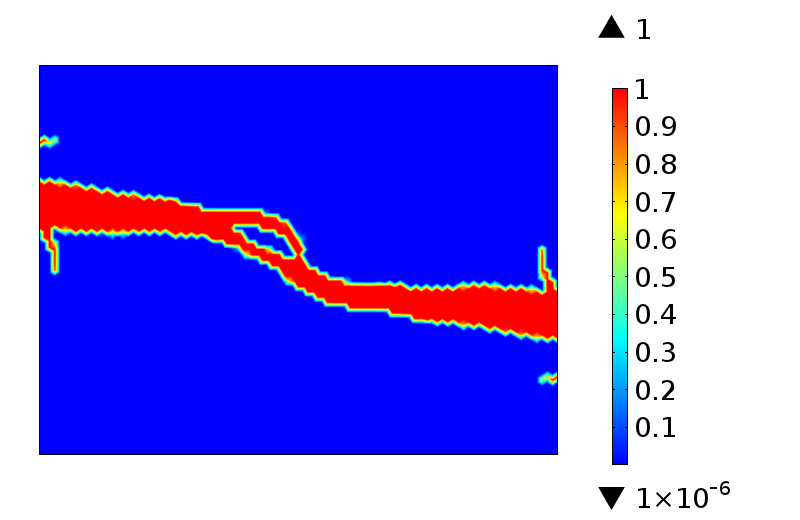}}
  \subfigure{\includegraphics[scale=0.25]{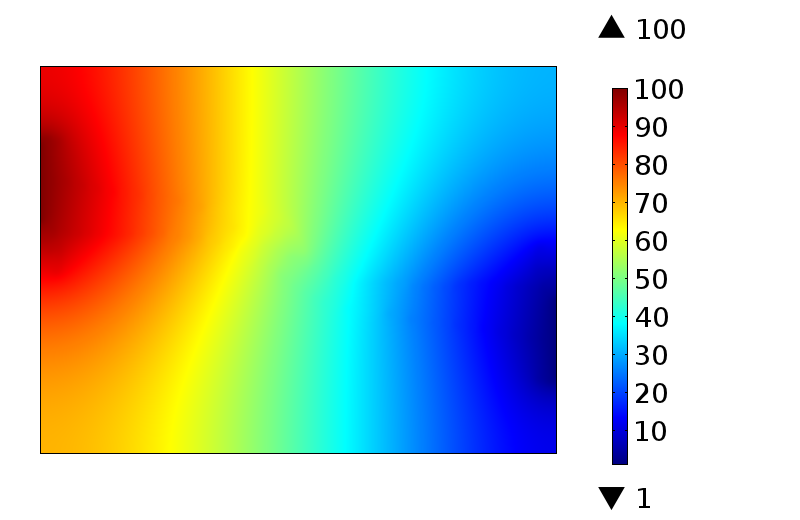}}
  \subfigure{\includegraphics[scale=0.25]{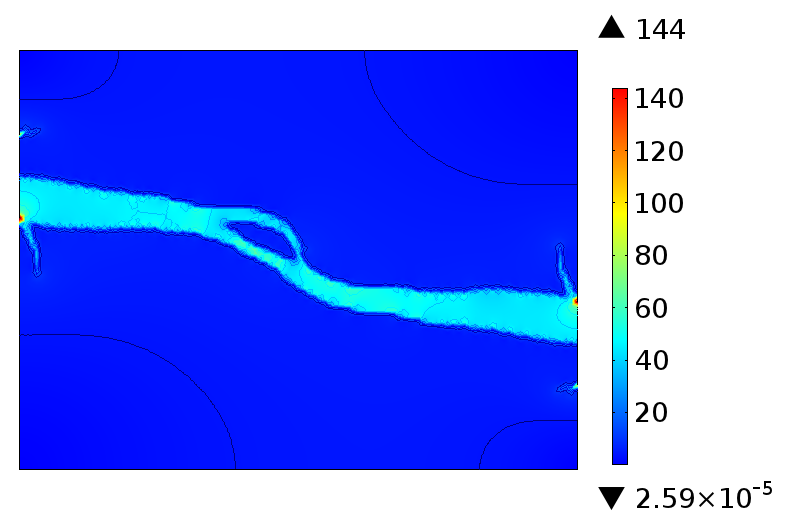}}
  
  {\small (a) Darcy model $\beta_B = 0$.}
  
  \subfigure{\includegraphics[scale=0.25]{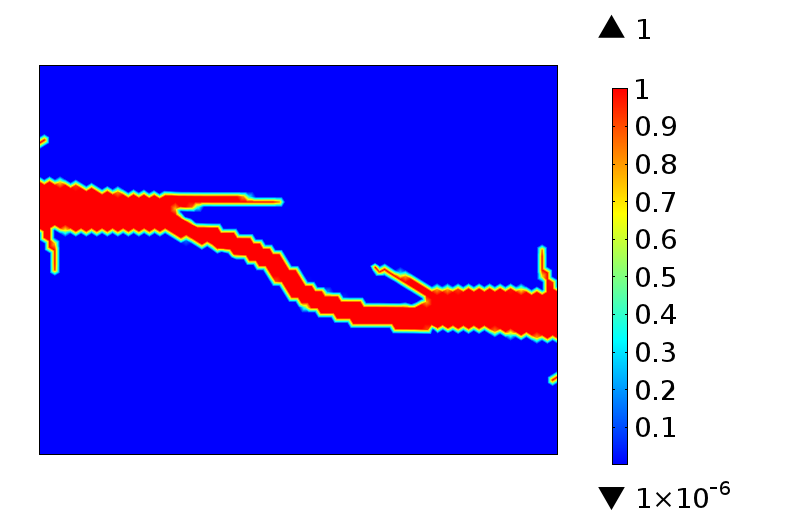}}
  \subfigure{\includegraphics[scale=0.25]{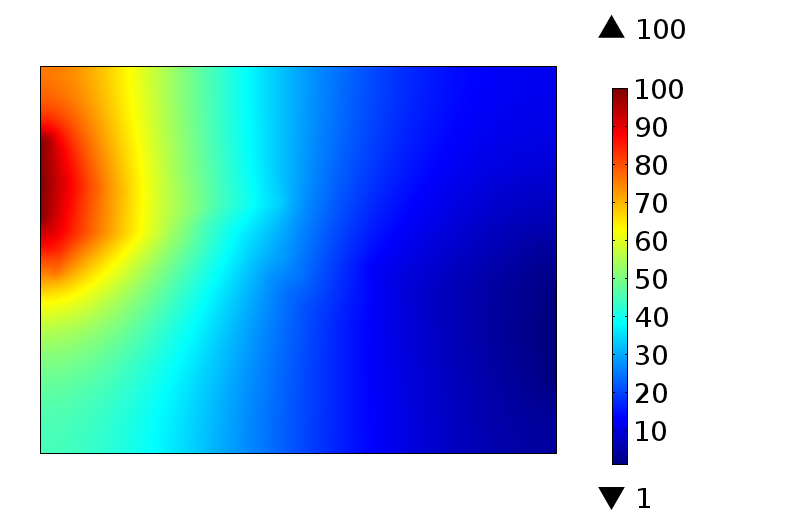}}
  \subfigure{\includegraphics[scale=0.25]{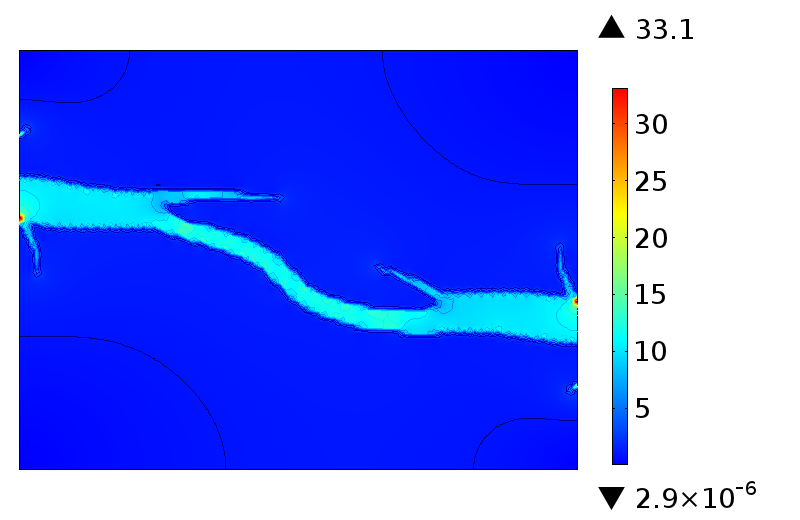}}
  
  {\small (b) Linearized Barus with $\beta_B = 0.1$.}
  
  \subfigure{\includegraphics[scale=0.25]{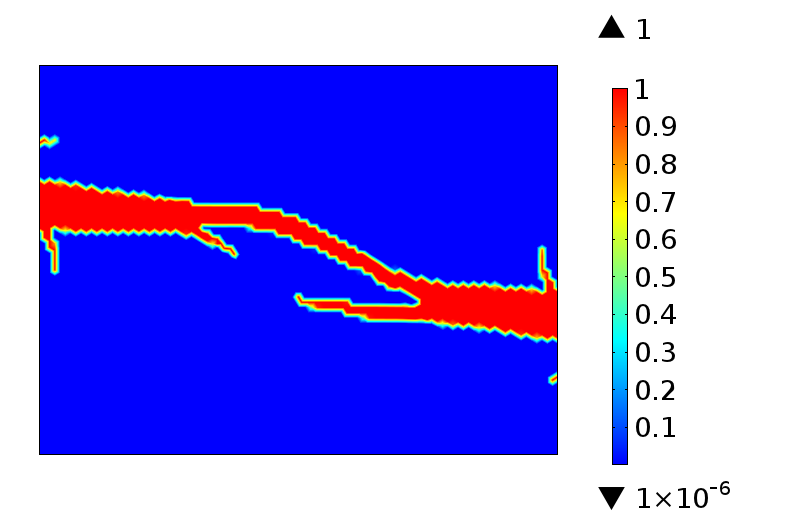}}
  \subfigure{\includegraphics[scale=0.25]{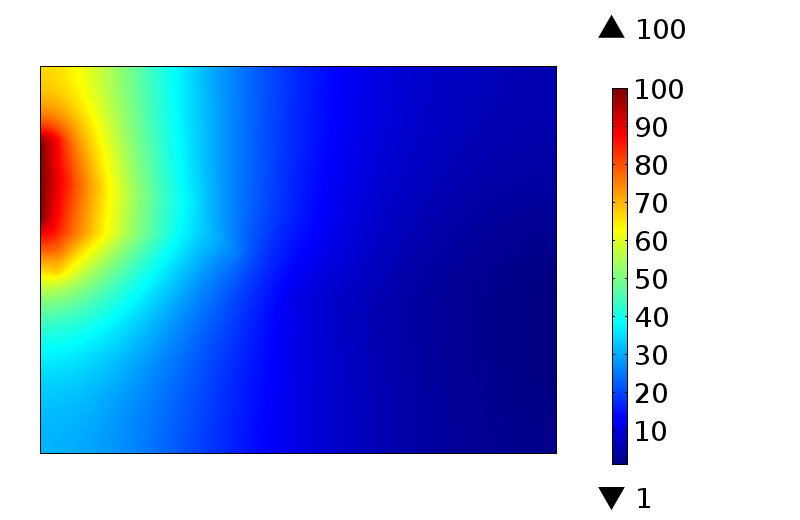}}
  \subfigure{\includegraphics[scale=0.25]{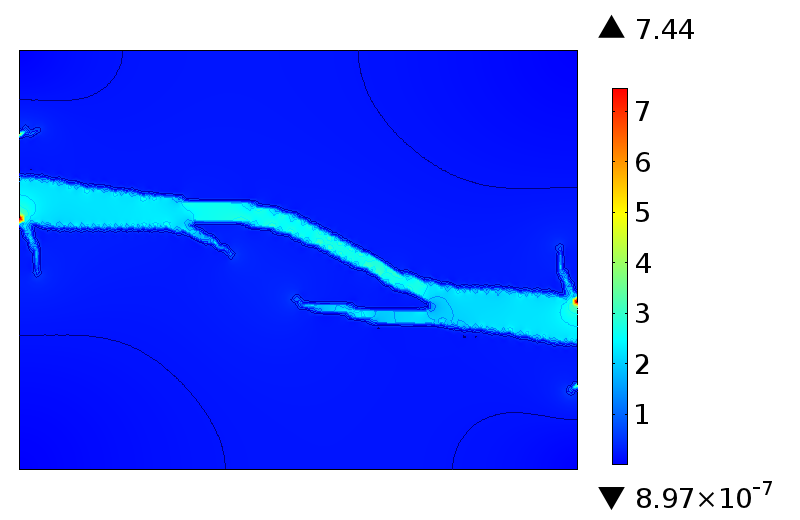}}
  
  {\small (c) Linearized Barus with $\beta_B = 0.75$.}
  
  \caption{\textsf{Pressure-driven problem with zero volumetric source:}
    This figure shows the optimized material distribution (left panel) and the corresponding pressure (middle) and velocity (right) profiles for various values of the Barus coefficient $\beta_B$. A volumetric bound constraint is placed on the high permeability material with $\gamma = 0.1$. In the material distribution plots, the high permeability material regions are represented in `red' color while `blue' color regions represent the low permeability. (See the online version for the color figure.) Three notable observations are: (i) The optimal material distribution remains almost the same for different Barus coefficient values. (ii) The magnitude of the velocity decreases with an increase in the Barus coefficient. (iii) The pressure spreads out lesser with an increase in the Barus coefficient. \label{Fig:Pressure_driven_zero_Q}}
\end{figure}

\begin{figure}[h]
	\subfigure{\includegraphics[scale=0.35]{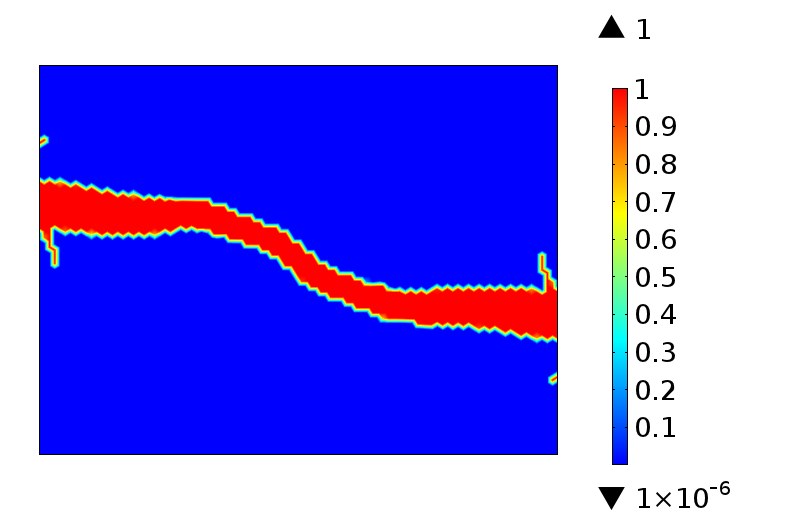}}
	\subfigure{\includegraphics[scale=0.35]{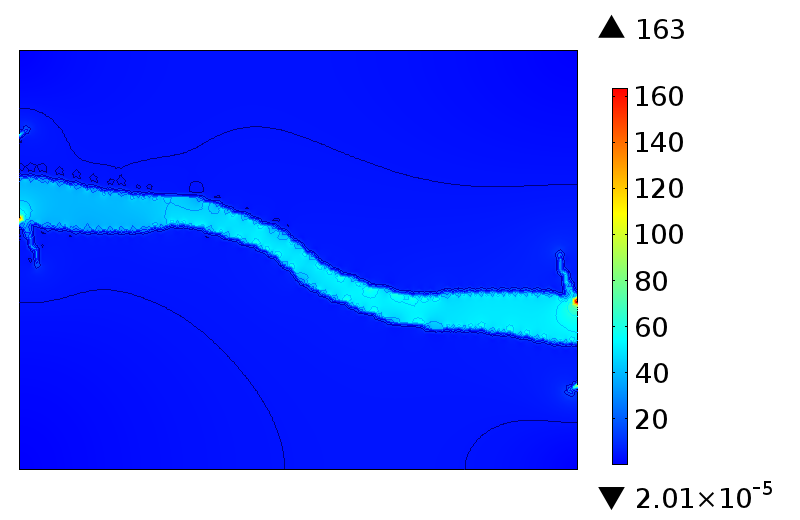}}

	{\small (a) Darcy model $\beta_B = 0$.}

	\subfigure{\includegraphics[scale=0.35]{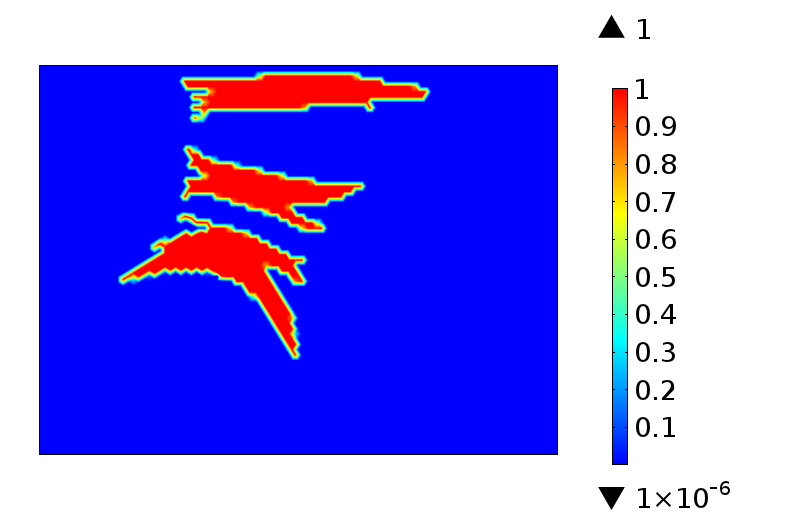}}
	\subfigure{\includegraphics[scale=0.35]{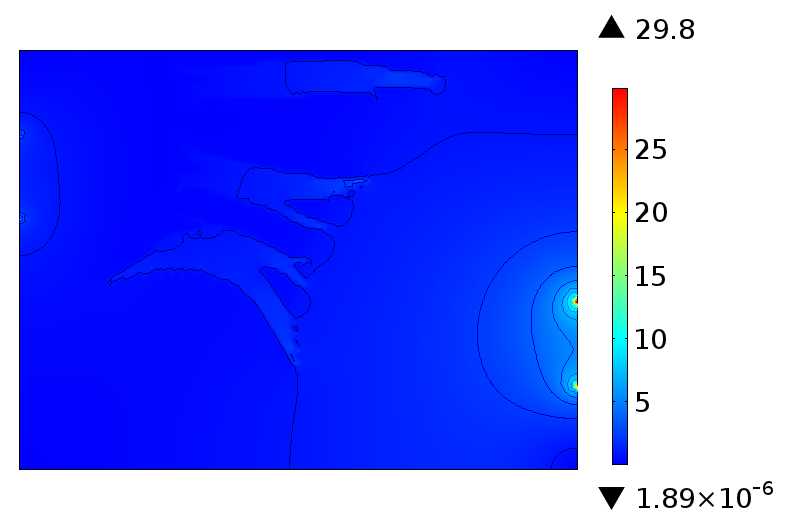}}

	{\small (b) Linearized Barus with $\beta_B = 0.125$.}

	\subfigure{\includegraphics[scale=0.35]{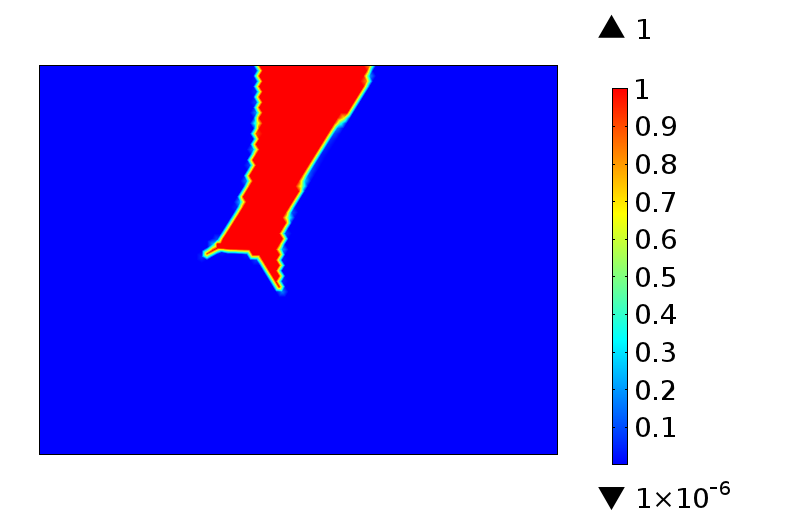}}
	\subfigure{\includegraphics[scale=0.35]{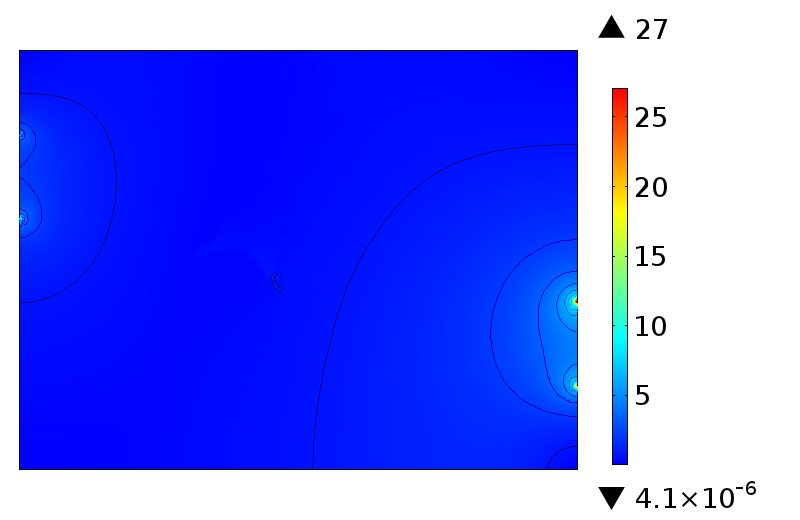}}

	{\small (c) Linearized Barus with $\beta_B = 0.25$.}

	\caption{\textsf{Pressure-driven problem with non-zero volumetric source:}
          This figure shows the optimal material distribution (left panel) 
	  and the corresponding velocity profile (right) for various values
          of the Barus coefficient $\beta_B$. The regions with high permeability
          material are shown in `red' while `blue' represents the low permeability
          material regions. (See the online version for the color figure.) The total
          area occupied by the high permeability material is limited to $\gamma = 0.1$.
          The value for the uniform volumetric source is $Q = 1$. Two notable differences
          compared to the case of zero volumetric source are (cf. figure \ref{Fig:Pressure_driven_zero_Q}):
          (i) The optimal material layout varies with the Barus coefficient.
          (ii) The material distribution under the linearized Barus model is
          not necessarily along the flow path. In comparison, under the Darcy
          model, the high permeability material is placed along the flow path
          of the fluid from inlet to the outlet.\label{Fig:Pressure_driven_nonzero_Q}}
\end{figure}

\begin{figure}[h]
	\subfigure[\small Material layout]{\includegraphics[scale=0.35]{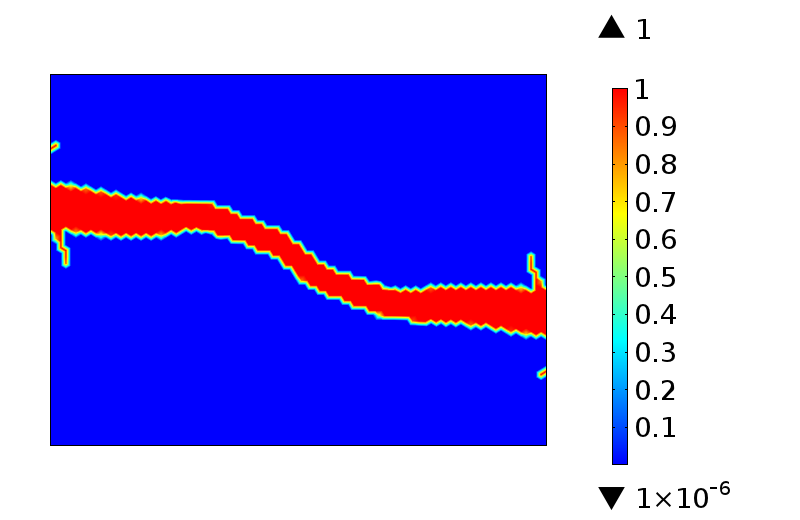}}
	\subfigure[Pressure profile]{\includegraphics[scale=0.35]{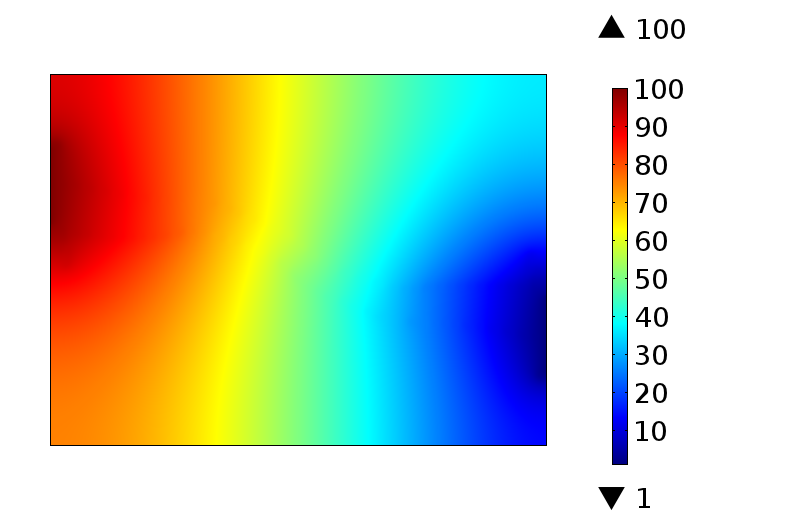}}
	\subfigure[Pressure gradient]{\includegraphics[scale=0.3]{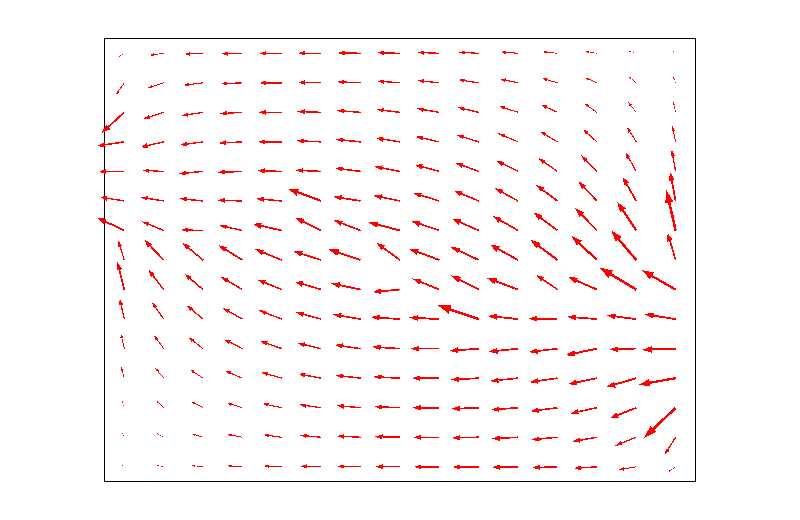}}
	\hspace{1 cm}
	\subfigure[Velocity profile]{\includegraphics[scale=0.35]{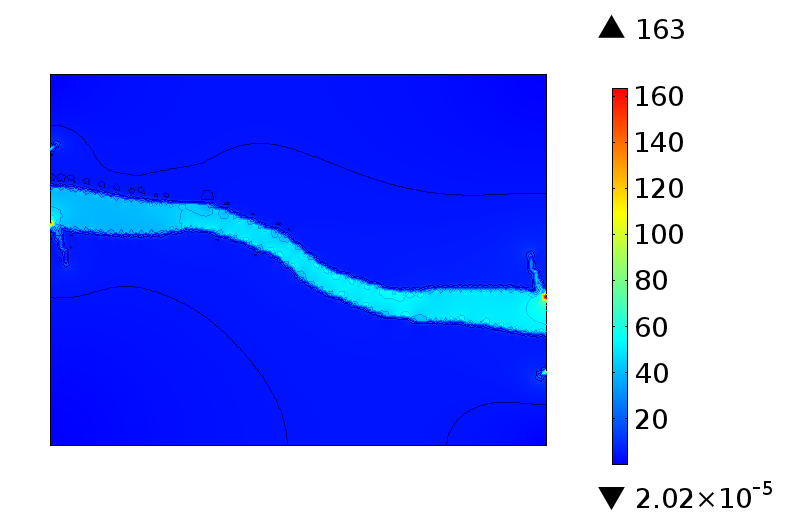}}
	\caption{\textsf{Pressure-driven problem with non-zero volumetric source (Darcy model):}
		This figure shows the optimal material distribution and the corresponding pressure profile, pressure gradients and velocity profile in the domain. The simulation parameters are: $\beta_B = 0$, $\gamma = 0.1$ and $Q = 1$. For the Darcy model, high permeability material (in 'red' color) is placed along high velocity and pressure gradients. (See the online version for color.) \label{Fig:Nonzero_Q_DarcyExplanation}}
\end{figure}

\begin{figure}[h]
  \subfigure[\small Material layout]{\includegraphics[scale=0.35]{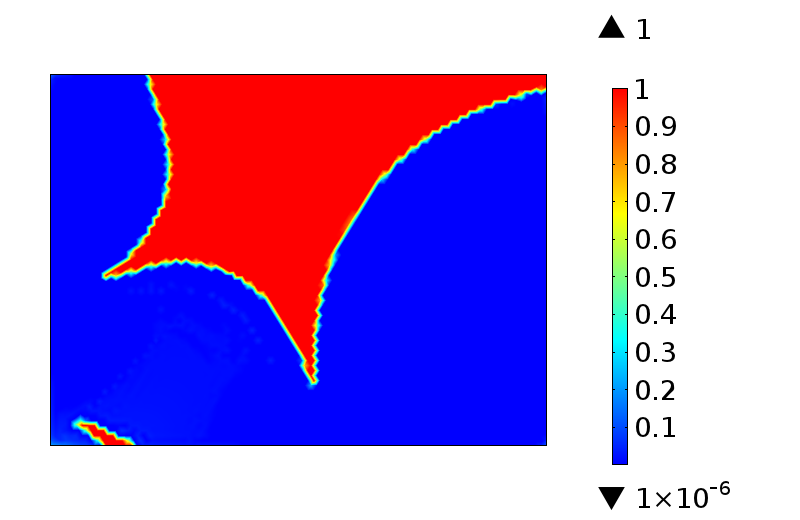}}
  \subfigure[Pressure profile]{\includegraphics[scale=0.35]{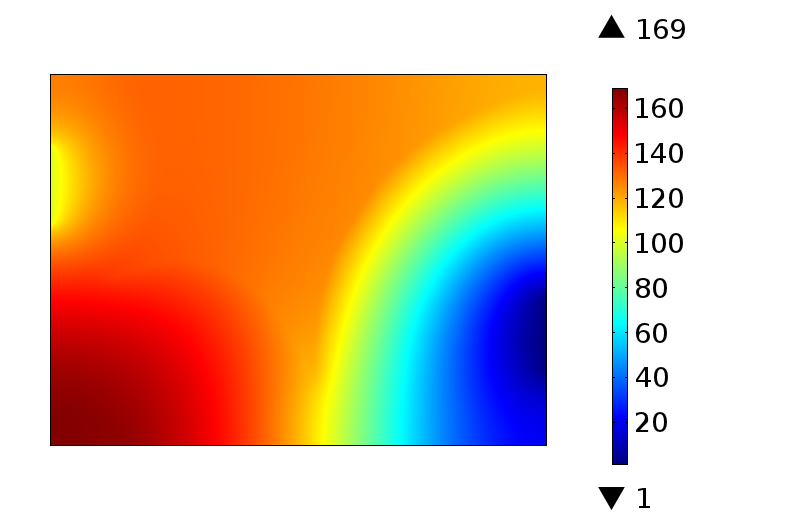}}
   \subfigure[Pressure gradient]{\includegraphics[scale=0.3]{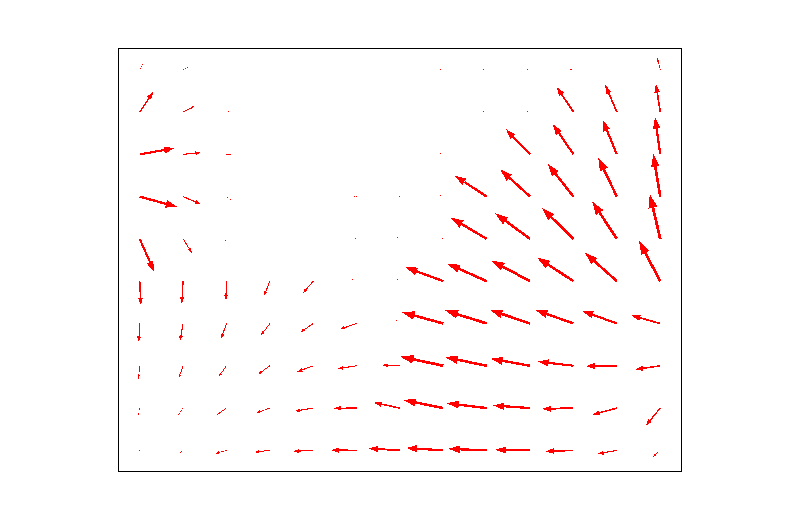}}
   	\hspace{1 cm}
  \subfigure[Velocity profile]{\includegraphics[scale=0.35]{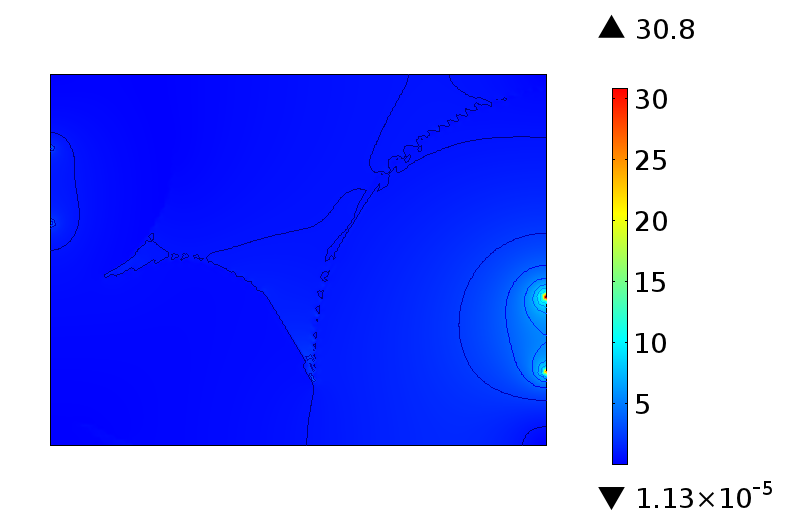}}
  \caption{\textsf{Pressure-driven problem with non-zero volumetric source (Darcy-Barus model):}
    This figure shows the optimal material distribution and the corresponding pressure profile, pressure gradients and velocity profile in the domain. The simulation parameters are: $\beta_B = 0.125$, $\gamma = 0.3$ and $Q = 1$. Under the optimal material layout, the regions with low pressure gradients have high-permeability material, indicated in `red' on the left panel. (See the online version for color.) \label{Fig:Nonzero_Q_BarusExplanation}}
\end{figure}


\subsection{Pipe-bend problem}
In order to ensure the above-mentioned trend is not specific to the previous problem, we will consider another canonical problem. Figure \ref{Fig:Pipe_bend_description} provides a pictorial description of this new problem. The difference between the previous and this problem lies in the location of the outlet (cf. figures \ref{Fig:Rect_domain2} and \ref{Fig:Pipe_bend_description}).

Figure \ref{Fig:Pipe_bend_sq_domain_Dist} shows the material distributions and associated velocity profiles for various values of the Barus coefficient. We also use this new problem to study the effect of aspect ratio on the optimal material layouts by solving the problem on two configurations -- a square domain ($a = b = 1.0$) and a rectangular domain ($a = 2.0$ and $b = 1.5$); see figure \ref{Fig:Pipe_bend_sq_vs_rect_Dist}. The following are the main observations:
\begin{enumerate}
\item Just like the pressure-driven problem, the optimal material distribution under the pipe-bend problem differs as the value of the Barus coefficient changes.
\item Under the Darcy model, the high-permeability material is placed in such a way as to provide maximum connectivity between the inlet and outlet. Specifically, the high-permeability is placed along the flow path, resulting in the form of a bent conduit from the inlet to the outlet. However, this trend is not observed under the linearized Barus model. Thus, due to the model's inherent non-linearity, one cannot guess the material placement under the (linearized) Barus model.
  \item Similar to the trend observed under the pressure-driven problem, the magnitude of the velocity decreases with an increase in the Barus coefficient.
  \item Under the Darcy model, keeping all the simulation parameters the same, \textsf{TopOpt} places material along the flow path for both square and rectangular domains. But the same is not true for the linearized Barus model. Because of the non-linearity, the material layout for a square domain cannot be extrapolated to a rectangular domain for the Barus model. Said differently, a na\"ive scaling will not work for material designs involving nonlinear models. 
\end{enumerate}

\begin{figure}[h]
  \includegraphics[scale=0.65]{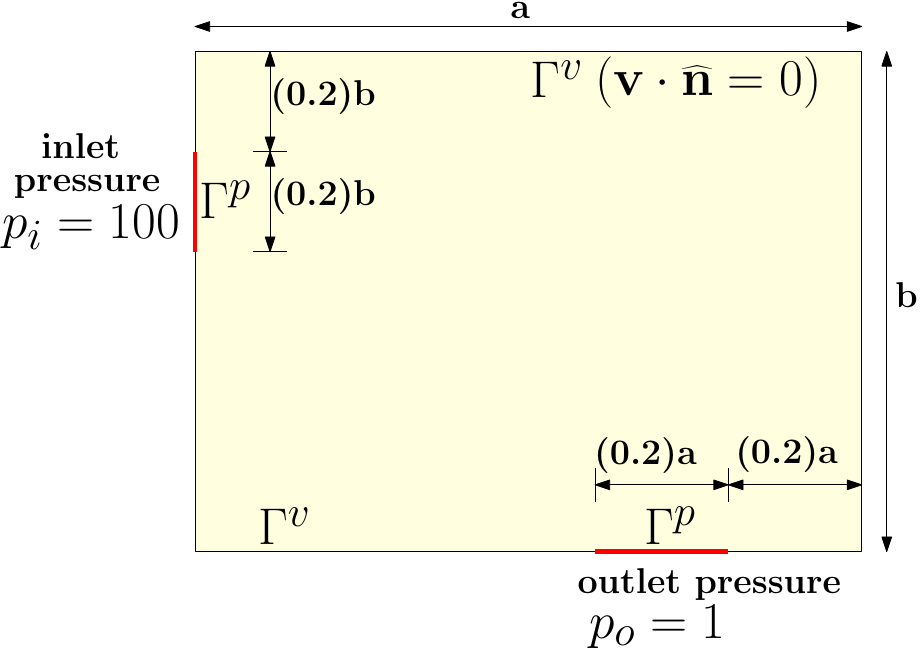}
  \caption{A square domain ($a = b = 1.0$) and a rectangular domain ($a = 2.0 \; \mathrm{and} \; b = 1.5$) are subjected to pressure boundary conditions with one inlet on left face and one outlet on bottom face. An inlet pressure $p_i = 100$ is applied on the left face of $\Gamma^{p}$ and an outlet pressure $p_o = 1$ is applied on the bottom face of $\Gamma^p$. At all other points on the boundary, the normal component of the velocity is prescribed to be zero. A uniform volumetric source $Q=1$ has been applied over the entire domain. \label{Fig:Pipe_bend_description}}
\end{figure}

 \begin{figure}[h]
 	\subfigure{\includegraphics[scale=0.35]{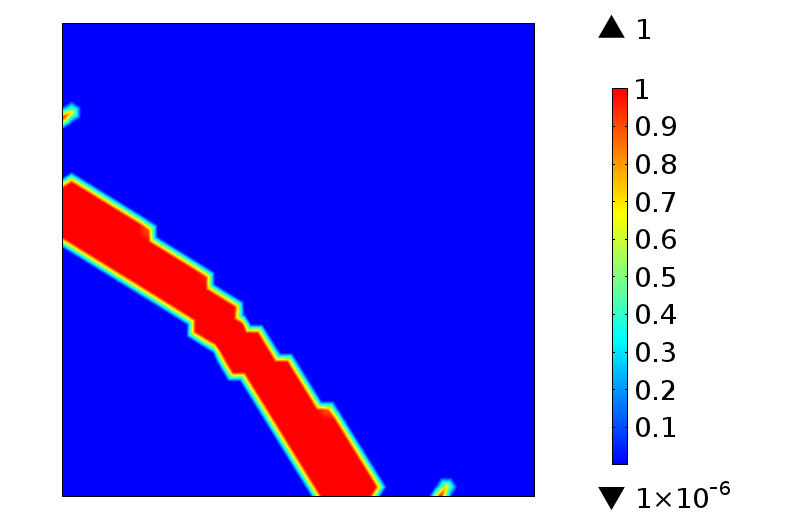}}
 	\subfigure{\includegraphics[scale=0.35]{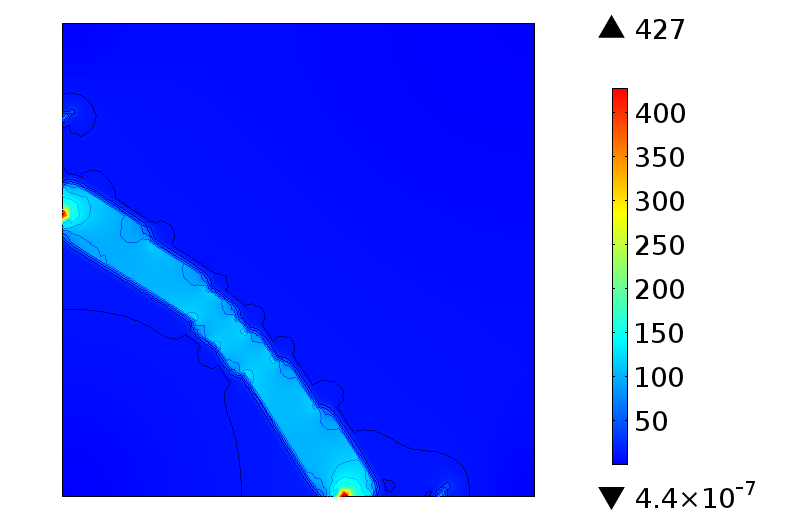}}

 	{\small (a) Darcy Model ($\beta_B$ = 0).}

 	\subfigure{\includegraphics[scale=0.35]{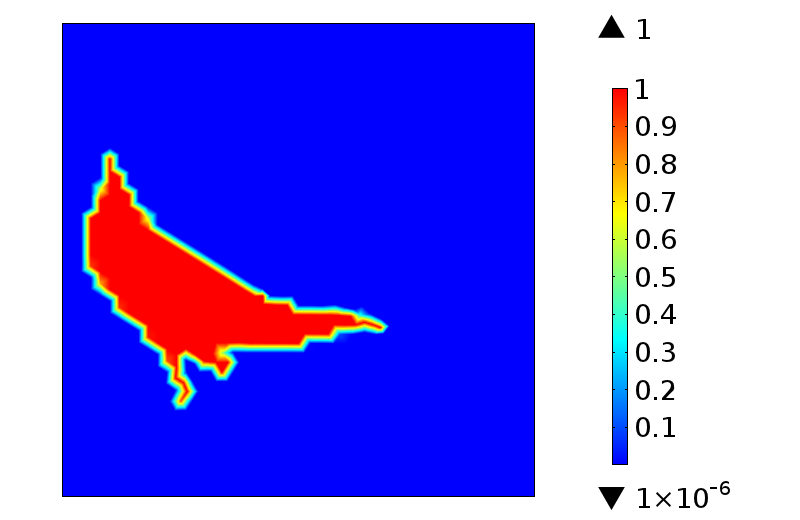}}
 	\subfigure{\includegraphics[scale=0.35]{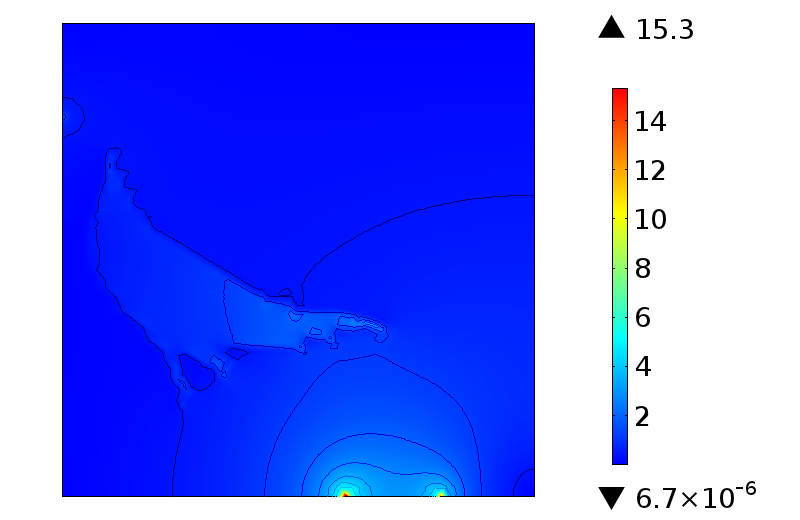}}

 	{\small (b) Linearized Barus with $\beta_B = 0.5$.}

 	\subfigure{\includegraphics[scale=0.35]{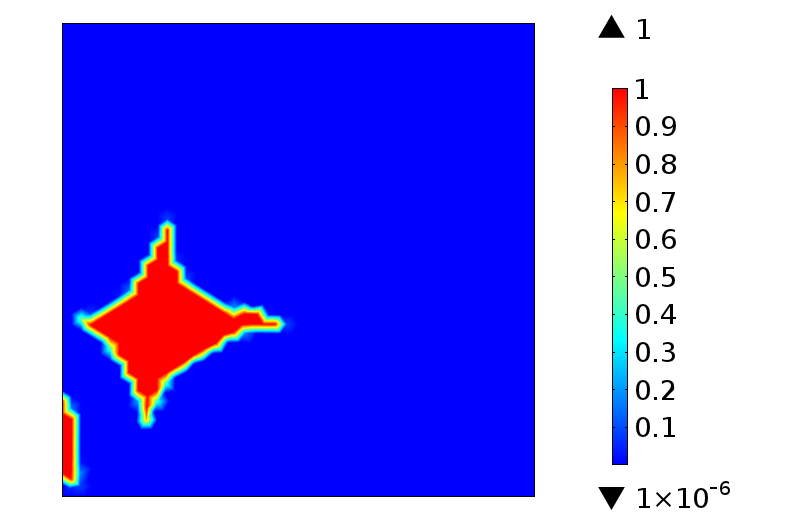}}
 	\subfigure{\includegraphics[scale=0.35]{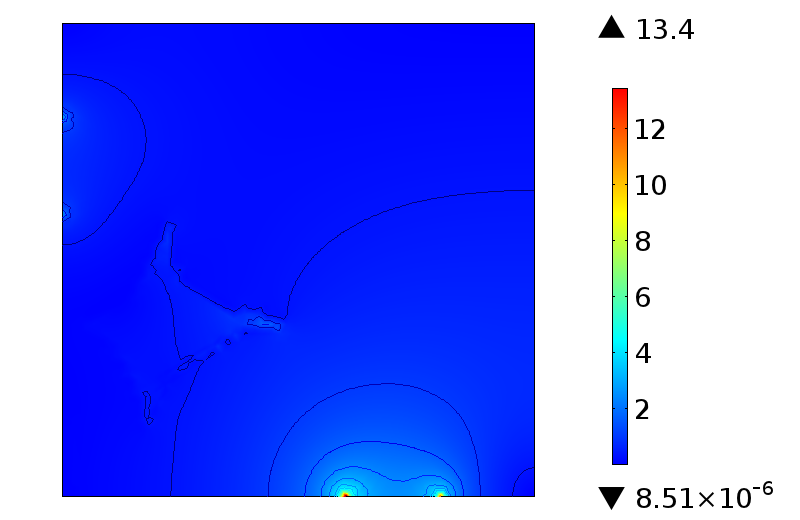}}

 	{\small (c) Linearized Barus with $\beta_B = 0.75$.}

 	\caption{\textsf{Pipe-bend problem (square domain):} This figure shows the optimal material distributions and the corresponding velocity profiles for various values of the Barus coefficient $\beta_B$. The computational domain is a square ($a = b = 1.0$). A uniform volumetric source $Q = 1$ is applied. The volumetric bound constraint of $\gamma = 0.1$ is placed on the high-permeability material. The regions occupied by the high-permeability material is represented by `red' while `blue' represents the low-permeability material. The main observations are: (i) just like the pressure-driven problem, the optimal material distribution varies with the Barus coefficient. (ii) The placement of the high-permeability material does not follow the flow path under the linearized Barus model (i.e., for non-zero $\beta_B$ values). (iii) The magnitude of the velocity decreases with an increase in the Barus coefficient.\label{Fig:Pipe_bend_sq_domain_Dist}}
 \end{figure}

\begin{figure}[h]
  \subfigure{\includegraphics[scale=0.32]{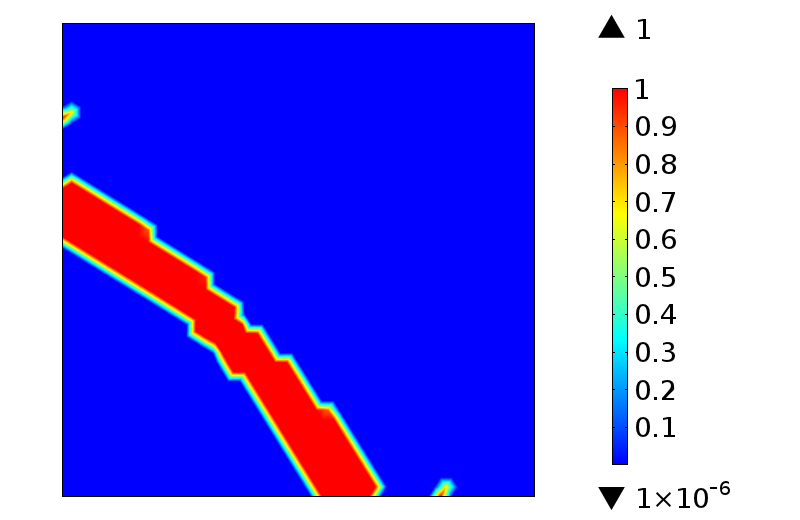}}
  \subfigure{\includegraphics[scale=0.35]{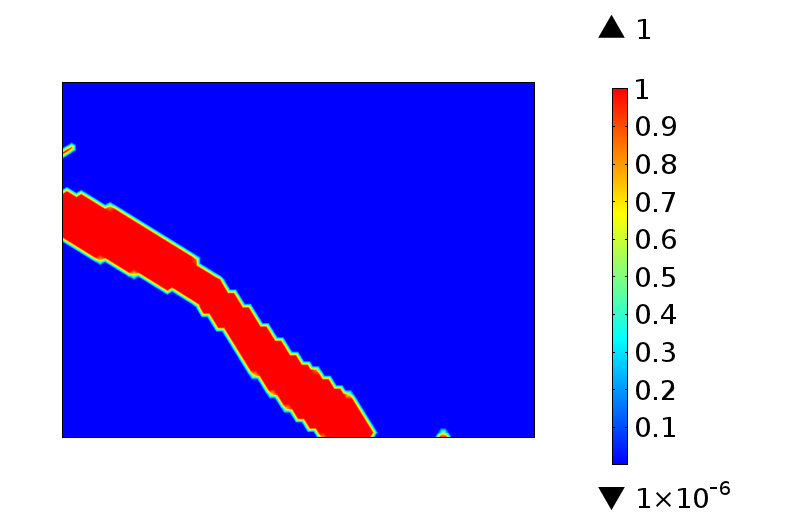}}
  
  {\small (a) Darcy model $\beta_B = 0$.}
  
  \subfigure{\includegraphics[scale=0.32]{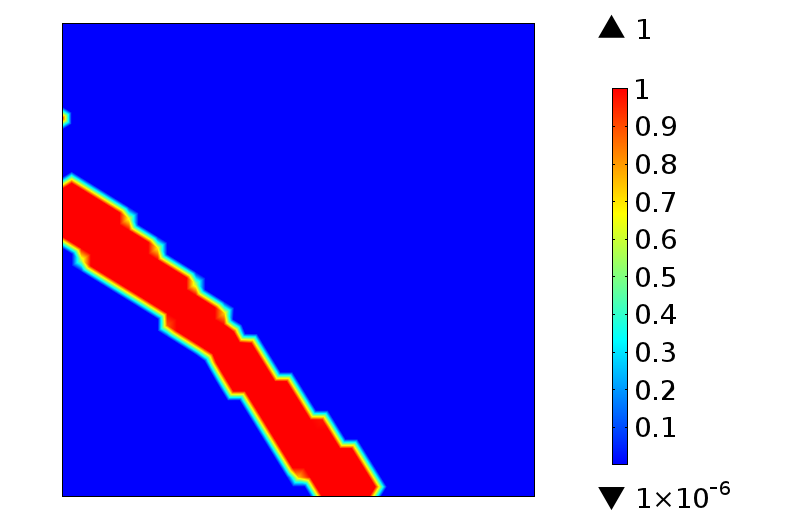}}
  \subfigure{\includegraphics[scale=0.35]{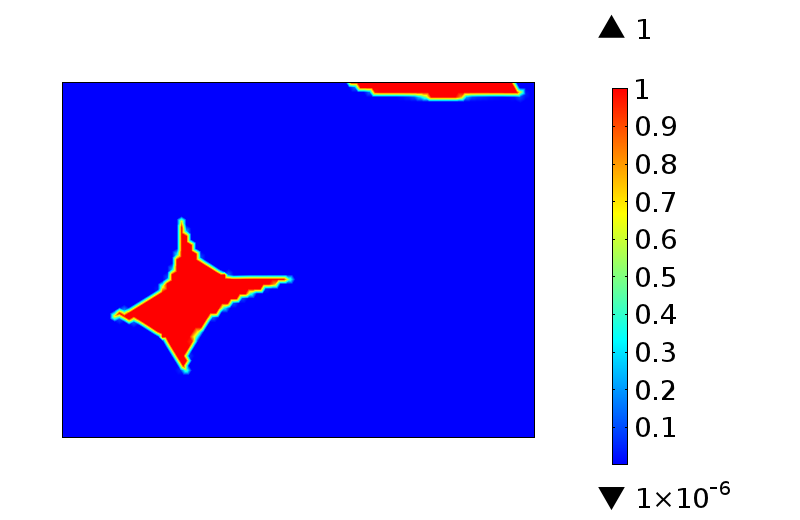}}
  
  {\small (b) Linearized Barus, $\beta_B$ = 0.25.}
  
  \subfigure{\includegraphics[scale=0.32]{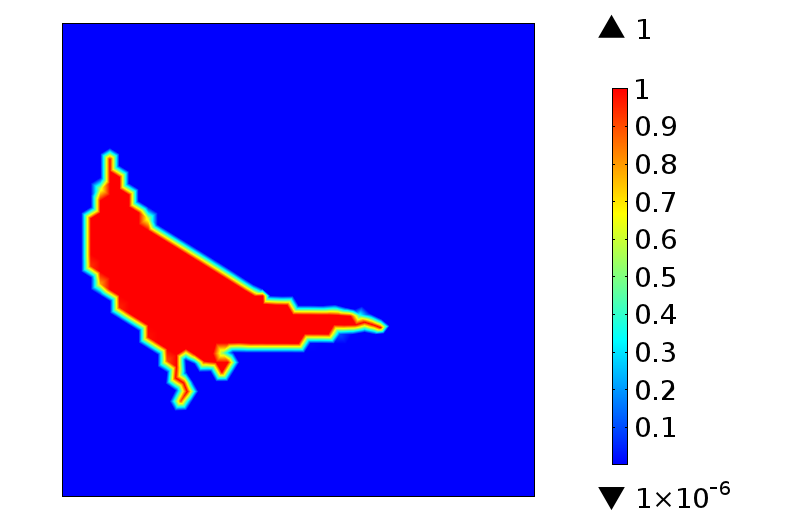}}
  \subfigure{\includegraphics[scale=0.35]{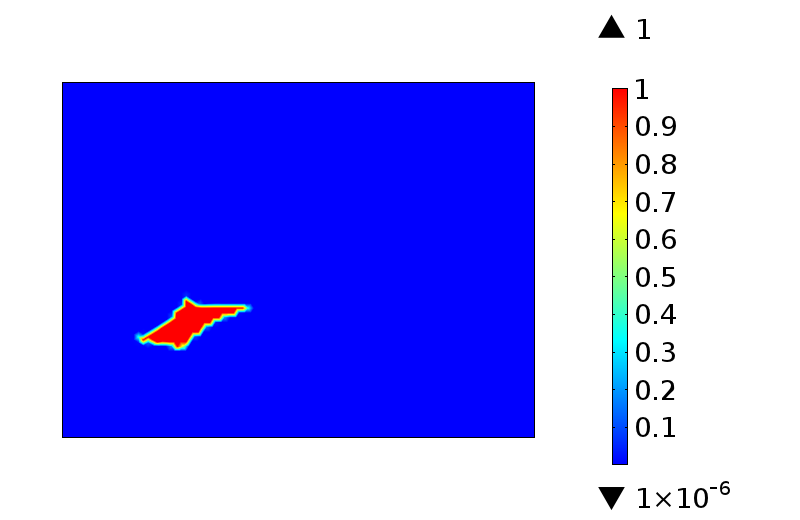}}
  
  {\small (c) Linearized Barus, $\beta_B$ = 0.5.}
  
  \caption{\textsf{Pipe-bend problem (square vs. rectangular domains):} This figure compares the optimal material layouts for a square domain (left panel) with that of a rectangular domain (right). All the other simulation parameters are the same. High permeability material is represented by `red' while `blue' represents the low permeability material. Under the Darcy model, the optimal material distributions are similar for both square and rectangular domains; the placement of the high-permeability material follows the flow path. But the same is not true for the linearized Barus model. Because of the non-linearity, the material layout for a square domain cannot be extrapolated to a rectangular domain for the Barus model.\label{Fig:Pipe_bend_sq_vs_rect_Dist}}
\end{figure}

\section{CLOSURE}
\label{Sec:S9_Porous_CR}
In this paper we have addressed obtaining optimal material layouts using topology optimization for flow through porous media applications. Some principal contributions of this article are: 
\begin{enumerate}[(C1)]
\item We have considered nonlinear models for flow through porous media that capture the pressure-dependent viscosity and inertial effects, thereby expanding the application space.
\item We proposed a well-posed design problem under \textsf{TopOpt} using the mechanical dissipation rate---a physical quantity with strong mechanics and thermodynamic underpinning. 
\item Although often used in the literature under \textsf{TopOpt}, the minimum power theorem has limited scope. Notably, we have shown that this theorem does not hold for nonlinear models describing the flow of fluids in porous media.
\item We have derived analytical solutions for 1D problems, and 2D and 3D problems with axisymmetric for different boundary conditions. These analytical solutions can be used to verify numerical simulators.
\item Using analytical and numerical solutions, we have studied the ramifications of the nonlinearities---pressure-dependent viscosity and inertial effects---on the optimal designs. 
\item We have addressed how the boundary conditions---pressure-driven versus velocity-driven---affect the formulation of the design problem. When should we pose the problem as minimization or maximization? Should we place the volume constraint on the material with higher permeability or lower?
\end{enumerate}

Our main findings are: 
\begin{enumerate}[(1)]
\item The analytical and numerical solutions using the dissipation rate have provided realistic material layouts, thereby establishing the appropriateness of the rate of dissipation to pose design problems under \textsf{TopOpt}.
\item In obtaining material designs of axisymmetric problems numerically, one should be wary of finger-like material layouts despite using the standard filters and SIMP regularization, especially when the primal analysis does not utilize the underlying symmetry. On the other hand, as shown in this paper, a primal analysis that invokes axisymmetry in its numerical scheme can provide a symmetric material design devoid of fingers. Next, for \textsf{TopOpt} involving flow through a porous annulus (either concentric cylinders or spheres), the high-permeability material is placed near the inner circle. This result will help the design of purifying water filters.
\item In the presence of a volumetric source, optimal designs under the Barus model are noticeably different from that of the Darcy model. 
\item We provide the following guiding points to devise well-posed design problems: 
\begin{enumerate}[(i)]
\item For pressure-driven problems, use the maximization of the rate of dissipation with a volume bound on the high-permeability material. 
\item For velocity-driven problems, use the minimization of the rate of dissipation with a volume bound on the high-permeability material. 
\end{enumerate}
\end{enumerate}

The proposed framework and the insights gained from the analytical and numerical solutions offer a way to achieve optimal material designs for devices with complicated  domains, new functionalities, and greater accuracy. A plausible future work can be towards using topology optimization for flow of multi-phase fluid systems and devices in which the porous solid is deformable.

\appendix
\section{MATHEMATICAL RESULTS}
\label{Sec:Appendix}

A function $f(x)$ is said to be convex if
\begin{align}
  \label{Eqn:TopOpt_convex_inequality}
  f(\lambda x_1 + (1 - \lambda) x_2)
  \leq \lambda f(x_1) + (1- \lambda) f(x_2)
  \quad \forall \lambda \in [0,1]
\end{align}
For twice-differentiable functions, a sufficient
condition for convexity is \citep{boyd2004convex}:
\begin{align}
  f^{''}(x)\geq 0 
\end{align}
where a superscript prime denotes a derivative.
Using this sufficient condition, it is easy to
check the functions $x^3$ and $1/x$ are convex
on the positive real line (i.e., $0 < x < \infty$).

\begin{lemma}
  \label{Lemma:Lemma_on_xi_stars}
  Given that $0 \leq \gamma \leq 1$ and $0 < r_i$,
  the quantities $\widehat{\xi}_{\mathrm{3D}}^{(1)}$ and
  $\widehat{\xi}_{\mathrm{3D}}^{(2)}$ defined in
  equations \eqref{Eqn:TO_porous_Sphere_gamma_in} and
  \eqref{Eqn:TO_porous_Sphere_gamma_out}, respectively,
  satisfy 
  \begin{align}
    1 + \frac{1}{r_i} -  \left( \frac{1}{\widehat{\xi}_{\mathrm{3D}}^{(1)}}
    + \frac{1}{\widehat{\xi}_{\mathrm{3D}}^{(2)}} \right) \geq 0
  \end{align}
\end{lemma}

\begin{proof} Noting that $f(x) = x^3$ is a convex
  function on the positive real axis and $\gamma
  \in [0,1]$, the definition of a convex function
  \eqref{Eqn:TopOpt_convex_inequality} implies the
  following:  
  \begin{align}
    ( \gamma (1) + (1 - \gamma) r_i )^3
    \; \leq  \; \gamma (1)^3 + (1 - \gamma) r_i^3
  \end{align}
  This implies that
  \begin{align}
    \frac{1}{\widehat{\xi}_{\mathrm{3D}}^{(1)}} 
    = \frac{1}{\sqrt[3]{\gamma + (1 - \gamma)r_i^3}}
    \leq \frac{1}{\gamma + (1 - \gamma)r_i} 
  \end{align}
  Using the fact that $f(x) = 1/x$ is a convex
  function on the positive real line, we establish
  the following inequality:
  \begin{align}
    \frac{1}{\widehat{\xi}_{\mathrm{3D}}^{(1)}} 
    \leq \frac{1}{\gamma + (1 - \gamma)r_i}
    \leq \gamma + (1 - \gamma) \frac{1}{r_i} 
  \end{align}
  By executing similar steps as above,
  we obtain the following inequality:
  \begin{align}
    \frac{1}{\widehat{\xi}_{\mathrm{3D}}^{(2)}}
    = \frac{1}{\sqrt[3]{(1 - \gamma) + \gamma r_i^3}}
    \leq \frac{1}{(1 - \gamma) + \gamma r_i}
    \leq (1 - \gamma) + \gamma \frac{1}{r_i} 
  \end{align}
  Adding the above two inequalities, we
  get the desired inequality: 
  \begin{align}
    \frac{1}{\widehat{\xi}_{\mathrm{3D}}^{(1)}} +
    \frac{1}{\widehat{\xi}_{\mathrm{3D}}^{(2)}}
    \leq 1 + \frac{1}{r_i}
  \end{align}
\end{proof}

\bibliographystyle{plainnat}
\bibliography{Master_References}
\end{document}